\documentclass{article}

\usepackage[T1]{fontenc}

\usepackage{xcolor}

\usepackage{amsmath, amssymb, amsfonts, amsthm, mathtools, mathrsfs, stmaryrd, dsfont}

\usepackage{graphicx}
\usepackage{tikz}
\usetikzlibrary{calc}
\usepackage{pgfplots}
\pgfplotsset{compat=1.18}

\usepackage{booktabs}
\usepackage{tabularx}
\usepackage{caption}
\usepackage{algorithmic}
\usepackage{algorithm}

\usepackage{enumitem}
\usepackage{changepage}
\usepackage{comment}

\usepackage{hyperref}
\usepackage{cleveref}

\usepackage{microtype}
\usepackage{csquotes}

\newtheorem{theorem}{Theorem}
\newtheorem{proposition}{Proposition}
\newtheorem{lemma}{Lemma}
\newtheorem{corollary}{Corollary}
\newtheorem{definition}{Definition}
\newtheorem*{maintheorem}{Main Theorem}

\newcommand{\RR}{\mathbb{R}}
\newcommand{\CC}{\mathbb{C}}

\newcommand{\PP}{\mathbb{P}}

\usepackage{subcaption}
\newtheorem{remark}{Remark}

\usepackage[
    backend=biber,
    style=alphabetic
]{biblatex}

\begin{document}

\title{Triangulation of Points Constrained to a Plane}

\author{
\large Petr Hrubý\textsuperscript{1}, Elima Shehu\textsuperscript{2} \\[0.5em]
\small
\textsuperscript{1}KTH Royal Institute of Technology, Stockholm, Sweden\\
\small
\textsuperscript{2}Otto-von-Guericke-University Magdeburg, Germany
}

\maketitle

\begin{abstract}
We study the set of image tuples arising from fixed cameras observing varying planar 3-dimensional point configurations. We derive a formula for the number of complex critical points of the triangulation problem, which seeks to reconstruct such configurations from noisy image data. Valid for an arbitrary number of views, this formula quantifies the intrinsic algebraic complexity of planar triangulation. We validate our theoretical findings through numerical experiments on both synthetic and real data, demonstrating that incorporating the planar incidence constraints leads to faster point reconstruction and improved accuracy compared to unconstrained triangulation.
\end{abstract}

\section{Introduction}
\label{sec:intro}

The structure-from-motion (SfM) pipeline takes a collection of images as input and aims to recover both the camera poses and the 3-dimensional structure of the scene. After detecting visual features in the images and establishing correspondences across multiple views, these correspondences are used to estimate the camera geometry, which in turn enables triangulation of the underlying scene.

In this work, we focus on the triangulation stage and, in particular, on point correspondences arising from world points constrained to a common plane; see Figure~\ref{fig:plane-triangulation}. \begin{figure}[H]
\centering
\resizebox{0.4\linewidth}{!}{%
\begin{tikzpicture}[scale=0.2, line cap=round, line join=round]

% -------------------------
% Key points
% -------------------------

% Image plane
\coordinate (I1) at (8,7);
\coordinate (I2) at (12,5);
\coordinate (I3) at (12,0);
\coordinate (I4) at (8,2);

% World plane (parallelogram)
\coordinate (A)  at (18,5);
\coordinate (B)  at (31,2);
\coordinate (D)  at (31,13);
\coordinate (E)  at (18,16);

% Define spanning vectors of the plane
\coordinate (u) at ($(B)-(A)$);
\coordinate (v) at ($(E)-(A)$);

% Image and world point
\coordinate (x) at (9,4.7);
\coordinate (X) at (23.2,10.4);

\coordinate (C)  at (2,2);
% -------------------------
% Image plane
% -------------------------
\fill[gray!10] (I1)--(I2)--(I3)--(I4)--cycle;
\draw[line width=1.8pt] (I1)--(I2)--(I3)--(I4)--cycle;

% -------------------------
% World plane
% -------------------------
\fill[gray!15] (A)--(B)--(D)--(E)--cycle;
\draw[line width=1.8pt] (A)--(B)--(D)--(E)--cycle;

% Grid on plane
\foreach \t in {0.2,0.4,0.6,0.8}{
  \draw[gray!70, line width=1pt] ($(A)+\t*(v)$) -- ($(B)+\t*(v)$);
  \draw[gray!70, line width=1pt] ($(A)+\t*(u)$) -- ($(E)+\t*(u)$);
}

% -------------------------
% Camera and ray
% -------------------------

\draw[line width=1.8pt] (C)--(x);
% \draw[line width=1.8pt] (x)--(X);
\draw[line width=1.5pt] (2,2) -- (x);
\draw[line width=1.5pt] (10,5) -- (10.5,5.2);
\draw[line width=1.5pt] (11.1,5.45) -- (X);

\fill[blue] (x) circle[radius=0.45];
\node[below right] at (x) {$x$};

% Highlight special X
\fill[blue] (X) circle[radius=0.45];
\node[below right] at (X) {$X$};

% Plane label
\node[gray!50] at ($(A)!0.9!(D)$) {$\Pi$};

\fill[red] (C) circle[radius=0.45];
\node[below left] at (C) {$C$};
\end{tikzpicture}}
\caption{Illustration of planar triangulation. A world point \(X\) constrained to a plane is observed by a pinhole camera \(C\), producing an image point \(x\).}
\label{fig:plane-triangulation}
\end{figure}
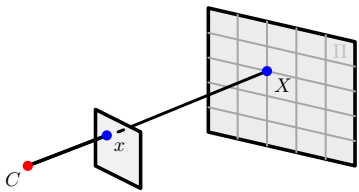
The main reason for this restriction is the simple observation that many real-world scenes encountered in practice are far from generic 3-dimensional configurations. Indoor environments and urban settings are dominated by planar surfaces such as walls, floors, ceilings, and facades, and the visual features extracted from these scenes, corners, junctions, and repeated architectural patterns are therefore frequently coplanar\cite{Holz2018,Li2020}. However, to the best of our knowledge, the geometric constraints induced by planarity are typically not explicitly exploited in standard triangulation pipelines. Treating all points as generic elements of \( \PP^3 \) overlooks available geometric structure and introduces unnecessary degrees of freedom into the reconstruction problem.

The main tools in our work are algebraic varieties, that is, the vanishing sets of systems of polynomial equations. Algebraic varieties have been used extensively to study triangulation problems arising from point correspondences, line correspondences, and minimal problems; see, for example, \cite{agarwal2019ideals,breiding2022line,kileel2022snapshot,Kileel_Thesis,Duff20_partialvis}. In these works, algebraic varieties arise naturally from the algebraic structure
of pinhole camera models.

A pinhole camera is modeled by a (complex) projective linear map \mbox{\(C:\PP^3\dashrightarrow\nolinebreak\PP^2\)} defined by a full-rank \( 3 \times 4 \) matrix.
A camera arrangement with \( m \geq 2 \) cameras is denoted by
\( \mathcal{C} = (C_1,\dots,C_m) \), and the associated joint camera map
\[
\Phi_{\mathcal{C}} :
\PP^3 \dashrightarrow (\PP^2)^m,
\qquad
X \longmapsto (C_1X,\dots,C_mX),
\]
models the process of imaging a world point \( X \) in homogeneous coordinates using \( m \) cameras.
For fixed cameras, the \emph{(point) multiview variety} \(\mathcal{M}_{\mathcal{C}}\) is defined as the smallest algebraic variety containing all image correspondences arising from~\( \Phi_{\mathcal{C}} \).

Our main contribution is the definition and study of planar-anchored point multiview varieties. Let \( \Pi \subset \PP^3 \) be a plane disjoint from all camera centers. Restricting the joint camera map to \( \Pi \) gives the rational map
\[
\Phi_{\mathcal{C}}|_{\Pi} :
\Pi \dashrightarrow (\PP^2)^m,
\qquad
X \longmapsto (C_1X,\dots,C_mX).
\]
The associated \emph{planar-anchored point multiview variety} is defined as the Zariski closure of the image,
\(
\mathcal{M}^\Pi_{\mathcal{C}} \;:=\; \overline{\Phi_{\mathcal{C}}|_{\Pi}(\Pi)} \;\subset\; (\PP^2)^m,
\)
i.e., the smallest algebraic variety in \( (\PP^2)^m \) containing all image correspondences arising from
points lying in \( \Pi \).
For these planar-anchored point multiview varieties, we derive explicit formula for their Euclidean distance degree (EDdeg), which serves as a measure of the algebraic complexity of error correction using exact algebraic methods; see, for example,\cite{draisma2016euclidean,breiding2024metric}. The main theoretical result of this work is the following theorem.

\begin{theorem}\label{EDdeg_plane-thm}
Let \( \mathcal{C} \) be a generic arrangement of \( m\geq 2 \) calibrated pinhole cameras. Then the Euclidean
distance degree of the affine planar-anchored point multiview variety $M^\Pi_{\mathcal{C}}$ satisfies
\[
EDdeg(M^\Pi_{\mathcal{C}})
= \frac{9}{2}m^2 - \frac{13}{2}m + 3.
\]
\end{theorem}
Using tools from applied algebraic geometry, we prove Theorem~\ref{EDdeg_plane-thm}, which was previously conjectured in \cite{DuffRydell2024}, where anchored varieties are studied in a general framework. While several properties of planar-anchored point multiview varieties can be viewed as special cases of results proved there, our approach focuses on the planar-anchored case alone and provides explicit coordinate descriptions and proofs.

From both geometric and algorithmic perspectives, the planar constraint has important consequences: restricting world points to a plane reduces the degrees of freedom of triangulation from three to two, simplifying the algebraic structure of the reconstruction problem. This reduction is reflected in the geometry of the associated multiview variety and in the complexity of the corresponding Euclidean distance minimization problem, suggesting that planar triangulation admits faster and more stable solvers than unconstrained 3D triangulation. Similar observations appear in \cite{RST23}, where points constrained to a line reduce the intrinsic difficulty of triangulation.

Building on this viewpoint, we complement our theoretical analysis with numerical experiments assuming complete visibility across calibrated pinhole cameras. These experiments confirm the predicted complexity and demonstrate strong practical performance on synthetic data in \texttt{HomotopyContinuation.jl}. We also develop solvers for real data, achieving more accurate triangulation than classical unconstrained methods.

\section{Planar-Anchored Point Multiview Varieties}
A \emph{camera} is a full-rank $3 \times 4$ matrix. 
A \emph{camera arrangement} is a collection 
\[
\mathcal{C} = (C_1, \ldots, C_m), \quad m \geq 2,
\]
of such cameras, with the requirement that their \emph{centers} (that is, their kernels) are all distinct.  

A \emph{variety} means the solution set of a system of polynomial equations. 
Given a set $U$, its \emph{Zariski closure} is the smallest variety containing $U$.  
We will work inside complex projective space, written $\mathbb{P}^n$, defined as
\[
\mathbb{P}^n = \big(\mathbb{C}^{n+1} \setminus \{0\}\big) / \sim,
\quad \text{where } x \sim y \;\; \text{whenever } x \text{ and } y 
\text{ differ by a nonzero scalar.}
\]

We write \( GL_3 \) for the group of invertible \( 3\times 3 \) complex matrices, and by~\( PGL_3\) the associated projective linear group, whose elements act on \( \mathbb{P}^2 \) up to nonzero scalar multiples.

For a rational map
\[
\Phi : W \dashrightarrow I,
\]
we write $\Phi|_\mathcal{R}$ to mean the restriction of $\Phi$ to any subset $\mathcal{R} \subseteq W$.

\begin{definition} \label{def:planar_anchored_variety} 
Let $X$ be a point lying on a plane $\Pi \subset \mathbb{P}^3$, 
where $\Pi$ contains none of the camera centers. \textit{The planar-anchored point multiview variety}, denoted~$\mathcal{M}_\mathcal{C}^\Pi$, 
is defined as the Zariski closure of the image of the map
\[
\Phi_\mathcal{C}|_\Pi : \Pi \dashrightarrow (\mathbb{P}^2)^m, 
\quad 
X \longmapsto (C_1 X, \ldots, C_m X).
\]
\end{definition} 

The planar-anchored point multiview variety $\mathcal{M}_\mathcal{C}^\Pi$ is the smallest variety that contains all point correspondences $\Phi_\mathcal{C}(X)$ for $X \in \Pi$ for $\Pi$ meeting no center.

\medskip

%-------------------------------------------------
We next describe the restricted image parametrically and show that it is characterized by planar homographies.
\begin{proposition} \label{prop:homographies}
Let $\mathcal{C}=(C_1,\dots,C_m)$ be cameras with centers $c_i$, and let~$\Pi\subset\PP^3$ 
be a plane with $c_i\notin\Pi$ for all $i$.
Choose $U_\Pi\in \mathbb C^{4\times3}$ full rank spanning $\Pi$, and set $A_i:=C_iU_\Pi\in GL_3$.
Then
\[
\operatorname{Im}(\Phi_\mathcal{C}|_\Pi)
=\{(x_1,\dots,x_m)\in(\PP^2)^m:\exists\,y\in\PP^2,\ x_i\sim A_i y\ \ \forall i\},
\]
and equivalently
\[
\overline{\operatorname{Im}(\Phi_\mathcal{C}|_\Pi)}
=\{(x_1,\dots,x_m)\in(\PP^2)^m:\ x_j\sim H_{ji}x_i\ \ \forall i,j\},
\]
where $H_{j i}:=A_jA_i^{-1}\in PGL_3$.
\end{proposition}

\begin{proof}
Since $\Pi$ is a plane in $\mathbb{P}^3$, every point $X \in \Pi$ can be written as
\[
X = U_\Pi y, \quad y \in \mathbb{P}^2,
\]
where $U_\Pi$ is a fixed $4 \times 3$ full-rank matrix spanning $\Pi$.  

For each camera $C_i$, the image of $X$ is
\[
x_i = C_i X = C_i U_\Pi y = A_i y,
\]
where $A_i := C_i U_\Pi \in GL_3,$  since \( c_i \notin \Pi \) implies that \( C_i|_\Pi \) is injective..  
Thus, the tuple $(x_1,\ldots,x_m)$ belongs to the image of~$\Phi_\mathcal{C}|_\Pi$ if and only if there exists some~$y \in \mathbb{P}^2$ such that
\[
x_i \sim A_i y \qquad \forall i.
\]
This proves the first description.

For the second description, suppose $x_i \sim A_i y$ and $x_j \sim A_j y$. Eliminating $y$, we obtain
\[
x_j \sim A_j A_i^{-1} x_i,
\]
The assumption $c_i \notin \Pi$ guarantees that $A_i = C_i U_\Pi$ is invertible in $PGL_3$.
Defining $H_{j i} := A_j A_i^{-1} \in PGL_3$, it follows that
\[
x_j \sim H_{j i} x_i \qquad \forall i,j.
\]
which is Zariski closed (defined by polynomials equivalent \mbox{to~$x_j\wedge H_{j i}x_i=0$).} Hence, the image of $\Phi_\mathcal{C}|_\Pi$ (or more precisely, its Zariski closure) is characterized as
\[
\overline{\mathrm{Im}(\Phi_\mathcal{C}|_\Pi)} 
= \{ (x_1,\ldots,x_m)\in(\PP^2)^m : x_j \sim H_{j i} x_i \ \ \forall i,j \}.
\]
The homography relations define polynomial equations, hence a Zariski closed subset of $(\mathbb{P}^2)^m$.
\end{proof}
The parametrization of the plane \( \Pi \) is unique up to projective linear automorphism. Any two choices \( U_\Pi \) and \( U_\Pi' \) differ by right multiplication with a matrix~\( G \in GL_3 \), i.e., \( U_\Pi' = U_\Pi G \), and \(A'_i=A_iG\). Hence, the resulting homography relations are invariant under the natural \( PGL_3 \)-action, \( A_j G (A_i G)^{-1} = A_j A_i^{-1} \).
%-------------------------------------------------
\subsection{Properties of Planar-Anchored Point Multiview Varieties}
A key feature of planar-anchored point multiview varieties is that they are
linearly isomorphic to multiview varieties obtained from projections
$\PP^2 \to \PP^2$. Here, a linear isomorphism refers to a map given by a
matrix whose inverse is also linear.

Let $\widehat{\mathcal C}=(\widehat C_1,\ldots,\widehat C_m)$ be an
arrangement of $m$ full-rank $3\times3$ matrices, and assume $m\ge 2.$
We denote by $\mathcal M^{2,2}_{\widehat{\mathcal C}}$ the Zariski closure
of the image of the map
\[
\Phi_{\widehat{\mathcal C}}:\PP^2 \longrightarrow (\PP^2)^m,
\qquad
X \longmapsto (\widehat C_1 X,\ldots,\widehat C_m X).
\]

\begin{theorem}
Let $\phi_{\Pi}:\Pi \to \mathbb P^2$ and
$\psi_{\mathcal C,i}:C_i(\Pi)\to\mathbb P^2$
be any choices of linear isomorphisms.
Let $\widehat{\mathcal C}$ denote the arrangement of matrices
\[
\widehat C_i
\;:=\;
\psi_{\mathcal C,i}\circ C_i \circ \phi_{\Pi}^{-1},
\qquad i=1,\ldots,m.
\]
Then
\[
\psi_{\mathcal C}
:=
(\psi_{\mathcal C,1},\ldots,\psi_{\mathcal C,m})
:\mathcal{M}_\mathcal{C}^\Pi
\longrightarrow
\mathcal M^{2,2}_{\widehat{\mathcal C}}
\]
is a linear isomorphism.
\end{theorem}

\begin{proof}
First note that $\Phi_{\mathcal C}^{\Pi}$ is well defined everywhere, since
$\Pi$ contains no camera center. Additionally, $\Phi_{\widehat{\mathcal C}}$ is defined everywhere. By construction, we have~\(\psi_{\mathcal C}(\Phi_{\mathcal C}|_{\Pi}(Y))
=
\Phi_{\widehat{\mathcal C}}(\phi_\Pi(Y))\),
which shows that $\psi_{\mathcal C}$ is a well-defined map.

\medskip

To prove surjectivity, let $x\in\mathcal M^{2,2}_{\widehat{\mathcal C}}$.
Then there exists $X\in\mathbb P^2$ such that~\mbox{\(x_i=\widehat C_i X\)} for all $i$.
Let $Y=\phi_\Pi^{-1}(X)\in\Pi$ and define
\(x'=\Phi_{\mathcal C}|_{\Pi}(Y)\), where \(x'_i=C_i(Y)\).
By construction, \(\psi_{\mathcal C}(x')=x\), which proves surjectivity.

\medskip

For injectivity, assume
\(\psi_{\mathcal C}(x)=\psi_{\mathcal C}(x')\).
Then for each $i$ we have~\mbox{\(C_i(Y)=C_i(Y')\).}
Since each restriction \(C_i|_\Pi:\Pi\to\mathbb P^2\) is a projective linear
isomorphism, it follows that \(Y=Y'\).
Therefore \(x=x'\), proving injectivity.
\end{proof}

\begin{remark}
This statement can also be viewed as a specialization of
\cite[Lemma~2.1]{DuffRydell2024} to the case
$\Lambda=\Pi\cong\mathbb P^2$.
\end{remark}
The linear isomorphisms constructed above allow us to relate planar-anchored
multiview varieties to the multiview varieties $\mathcal M^{2,2}_{\widehat{\mathcal C}}$. In particular, results established for multiview varieties arising from projections $\PP^2 \to \PP^2$ carry over directly to
$\mathcal{M}_\mathcal{C}^\Pi$. This observation will allow us to derive several structural properties of planar-anchored point multiview varieties by working directly with $3\times3$ matrix arrangements.

\begin{proposition}\label{bireg.isomorphism} 
Let $\mathcal{C} = (C_1,\dots,C_m)$ be a camera arrangement of pinhole cameras, and let $\Pi\subset\PP^3$ be a plane
containing no center. Then the planar-anchored point multiview variety
$\mathcal{M}_\mathcal{C}^\Pi$ is irreducible and
\[
\mathcal{M}_\mathcal{C}^\Pi\ \cong\ \PP^2.
\]
In particular, $\dim\mathcal{M}_\mathcal{C}^\Pi=2$.
\end{proposition}

\begin{proof}
Choose a full-rank matrix $U_\Pi\in\CC^{4\times 3}$ whose column span equals $\Pi$. For each $i$, set
\[
A_i := C_i U_\Pi \in \CC^{3\times 3}.
\]
Since $c_i\notin\Pi$, the restriction $C_i|_\Pi:\Pi\to\PP^2$ is a
projective linear isomorphism, hence $A_i\in GL_3$.

Consider the morphism
\[
\Phi_{\mathcal C}|_\Pi:\PP^2 \longrightarrow (\PP^2)^m,
\qquad
y \longmapsto (A_1y,\dots,A_my),
\]
where we identify $\Pi \cong \PP^2$ via $y\mapsto U_\Pi y$.
By definition,
\[
\mathcal{M}_\mathcal{C}^\Pi = \overline{\mathrm{Im}(\Phi_{\mathcal C}|_\Pi)}.
\]
Since $\PP^2$ is irreducible and $\Phi_{\mathcal C}|_\Pi$ is a morphism,
$\mathrm{Im}(\Phi_{\mathcal C}|_\Pi)$ is irreducible, and therefore its Zariski closure $\mathcal{M}_\mathcal{C}^\Pi$ is irreducible.

Moreover, $\Phi_{\mathcal C}|_\Pi$ is injective because $A_1\in GL_3$, and on its image it admits the regular inverse
\[
(x_1,\dots,x_m)\longmapsto A_1^{-1}x_1.
\]
Hence $\Phi_{\mathcal C}|_\Pi$ is a biregular isomorphism from $\PP^2$ onto $\mathrm{Im}(\Phi_{\mathcal C}|_\Pi)$. Since it is biregular onto its image, the image is isomorphic to a projective variety, hence closed. In particular, $\mathrm{Im}(\Phi_{\mathcal C}|_\Pi)$ is Zariski closed and thus equals
$\mathcal{M}_\mathcal{C}^\Pi$. Therefore $\mathcal{M}_\mathcal{C}^\Pi \cong \PP^2,$ thus~$\dim \mathcal M_C^\Pi = \dim \PP^2 = 2.$
\end{proof}

%-------------------------------------------------
\begin{proposition}
The planar-anchored point multiview variety $\mathcal{M}_\mathcal{C}^\Pi$ is smooth.
In fact, $\mathcal{M}_\mathcal{C}^\Pi\cong\PP^2$ is a smooth projective surface.
\end{proposition}

\begin{proof}
Immediate, since $\mathcal{M}_\mathcal{C}^\Pi\cong\PP^2$ by the previous proposition.
\end{proof}

%-------------------------------------------------
\begin{corollary}[Plane–induced homographies]
Let notation be as above. For $(x_1,\dots,x_m)\in(\PP^2)^m$, the following are equivalent:
\begin{enumerate}
\item $(x_1,\dots,x_m)\in\mathcal{M}_\mathcal{C}^\Pi$.
\item $\exists\,y\in\PP^2$ with $x_i\sim A_i y$ for all $i$.
\item $x_j\sim H_{j i}x_i$ for all $i,j$.
\end{enumerate}
Moreover, the homographies satisfy
\[
H_{i i}=I,\quad
H_{j i}^{-1}=H_{i j},\quad
H_{k j}H_{j i}=H_{k i}.
\]
\end{corollary}

\subsection*{Multidegree}
The \emph{multidegree} of the planar-anchored point multiview variety $\mathcal{M}^\Pi_{\mathcal{C}}$ is defined by the following function
\[
D_{\mathcal{M}^\Pi_{\mathcal{C}}}(d_1,\ldots,d_m) := 
\#\bigl(\mathcal{M}^\Pi_{\mathcal{C}} \cap (L^{(1)}_{d_1}\times \cdots \times L^{(m)}_{d_m})\bigr),
\]
for $(d_1,\ldots,d_m)\in\mathbb{N}^n$ such that $d_1+\cdots+d_m = \dim \mathcal{M}^\Pi_{\mathcal{C}}=2$, 
where for each~$1\le i\le m$ we denote by $L^{(i)}_d \subset \mathbb{P}^2$ a general linear subspace of codimension~$d$.

Following the fact that the function $\mathcal{M}^\Pi_{\mathcal{C}}$ is symmetric, meaning that 
\[
D_{\mathcal{M}^\Pi_{\mathcal{C}}}(d_1,\ldots,d_m) = 
D_{\mathcal{M}^\Pi_{\mathcal{C}}}(d_{\sigma(1)},\ldots,d_{\sigma(m)})
\]
for any permutation $\sigma$ on $m$ elements, we only have to determine the following two values:
$D_{\mathcal{M}^\Pi_{\mathcal{C}}}(2,0,\ldots,0)$ and 
$D_{\mathcal{M}^\Pi_{\mathcal{C}}}(1,1,0,\ldots,0)$.

\begin{theorem}
Up to permutation, the multidegree of the generic planar-anchored point multiview variety $\mathcal{M}^\Pi_{\mathcal{C}}$ is given by
\[
D_{\mathcal{M}^\Pi_{\mathcal{C}}}(2,0,\ldots,0)=1, 
\qquad 
D_{\mathcal{M}^\Pi_{\mathcal{C}}}(1,1,0,\ldots,0)=1.
\]
\end{theorem}

\begin{proof}
Let us write $p=(p_1,\ldots,p_m)$ for a point $p\in \mathcal{M}^\Pi_{\mathcal{C}}$. For the first value~$D_{\mathcal{M}^\Pi_{\mathcal{C}}}(2,0,\ldots,0)$, let us fix a generic point $p$ on the first image, which determines the back-projected line~$L$. We can see that the generic back-projected line $L$ meets generically the plane~$\Pi$ in just one point $X\in \mathbb{P}^3$. So there is only one point~$X$, one element when counting, implying the multidegree~$D_{\mathcal{M}^\Pi_{\mathcal{C}}}(2,0,\ldots,0)=1$.

Consider the second value $D_{\mathcal{M}^\Pi_{\mathcal{C}}}(1,1,0,\ldots,0)$. 
We fix two generic image lines,~$\ell_1,\ell_2$, each containing $p_1,p_2$ respectively, on the initial two images. The image lines $\ell_i$ determine two back-projected planes, $H_1, H_2$. When considering the intersection of these two generic planes with $\Pi$, we observe that, generically, their intersection consists of a single point. Consequently, the count equals one, that is,
\(
D_{\mathcal{M}^\Pi_{\mathcal{C}}}(1,1,0,\ldots,0)=1.
\)
\end{proof}
In particular, a correspondence arising from a point on $\Pi$ is uniquely determined (up to scale) by any one of its image points.

\section{The Euclidean Distance Problem}

Let \(\mathcal{V} \subset \RR^N\) be an affine chart of a planar-planar-anchored point multiview variety and let \(u \in \RR^N\) represent noisy image data. The Euclidean distance problem is
\begin{equation}
\min_{x \in \mathcal{V}} \|x - u\|^2. \label{eq:EDdeg_problem}
\end{equation}

The number of complex critical points of this optimization problem is called the \emph{Euclidean distance degree}(EDdeg) of \(\mathcal{V}\). This invariant measures the algebraic complexity of fitting noisy data to the planar model and serves as a proxy for the difficulty of triangulation via exact algebraic methods.

\begin{figure}
    \centering
    \includegraphics[width=0.52\linewidth]{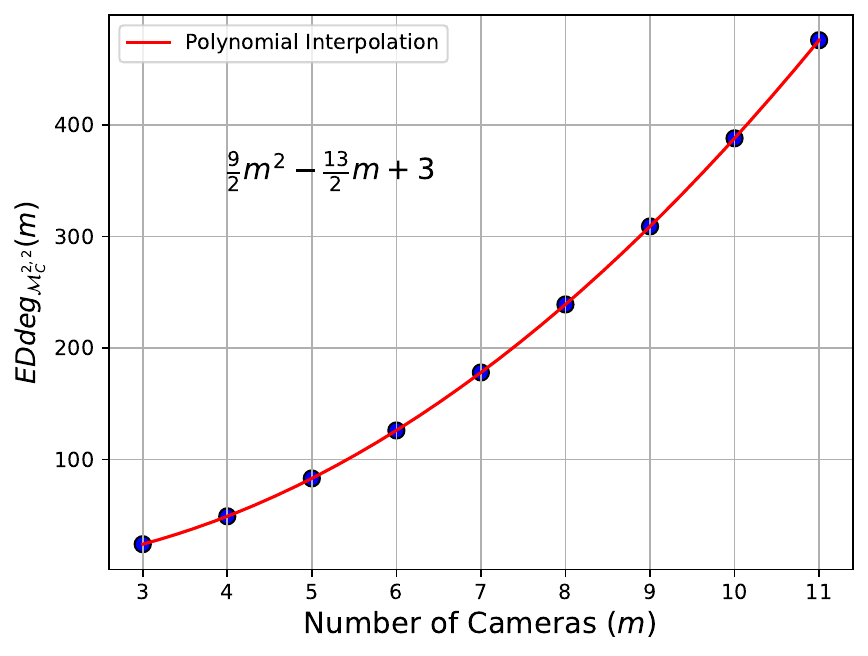}
    \caption{Computed EDdeg values for the affine planar-anchored point multiview variety in \texttt{HomotopyContinuation.jl}~\cite{HC.jl} as a function of the number of cameras, together with a polynomial fit.}
    \label{fig:placeholder}
\end{figure}
Let $\Pi \subset \mathbb{P}^3$ be a fixed plane and $m \geq 2$ for the planar-anchored point multiview variety $\mathcal{M}^\Pi_{\mathcal{C}}$. We now restate the main theorem from the introduction.
\begin{theorem}\label{EDdeg_plane-thm}
    Let $\mathcal{C}$ be a generic arrangement of $m\geq 2$ cameras, then the Euclidean distance degree of the affine planar-anchored point multiview variety~$M^\Pi_{\mathcal{C}}$~is
    \[
    EDdeg(M^\Pi_{\mathcal{C}})= \frac{9}{2}m^2-\frac{13}{2}m+3.
    \]
\end{theorem}
Computed ED degree values (Figure~\ref{fig:placeholder}) exhibit a quadratic trend, motivating the formula in Theorem~\ref{EDdeg_plane-thm}. The proof of this theorem is provided in the supplementary material, which we refer to from now on as SM.
\section{Solvers}\label{sec:solvers}

In this section, we introduce practical solvers for triangulating 3D points constrained to a plane. We consider the cases with $m=2$ and $m=3$ views.

Let us have $m \geq 2$ cameras $\mathcal{C}=(C_1,...,C_m)$, a plane $\Pi$, and a set of 2D observations $(u_1,...,u_m) \in (\PP^2)^m$.
To optimally triangulate a 3D point $X \in \Pi$ from the 2D observations $u = (u_1,...,u_m)$, we search for the point $x \in (\PP^2)^m$ on the planar-anchored point multiview variety $\mathcal{M}_\mathcal{C}^\Pi$ (Definition~\ref{def:planar_anchored_variety}), minimizing the Euclidean distance \eqref{eq:EDdeg_problem}. According to Theorem~\ref{EDdeg_plane-thm}, this problem can be transformed into a polynomial system with $\frac{9}{2}m^2 - \frac{13}{2}m + 3$ solutions. In particular, we get $8$ solutions for $m=2$ and $24$ solutions for $m=3$.

According to Proposition~\ref{prop:homographies}, the points on the variety $\mathcal{M}_\mathcal{C}^\Pi$ are in a one-to-one correspondence with the projections $x_1 \in \PP^2$ onto the first camera $C_1$. Moreover, there exist homographies $H_{j,1}, j \in \{2,\ldots,m\}$, such that $x_j \sim H_{j,1} x_1$. Homography $H_{j,1}$ can be calculated from the cameras $C_1, C_j$ and the plane $\Pi$.

Let $q: \RR^3 \dashrightarrow \RR^2, [x \ y \ z ]^T \longmapsto [\frac{x}{z} \ \frac{y}{z}]$ be a map sending homogeneous coordinates onto the corresponding affine coordinates.
Then, the affine chart of $\mathcal{M}_\mathcal{C}^\Pi$ can be parameterized (almost everywhere) with two variables $a,b \in \RR$, such that $x_1(a,b) = [a \ b \ 1]^T$, and
\begin{equation}
    x(a,b) = (x_1(a,b), q(H_{2,1}x_1(a,b)),\ldots,q(H_{m,1}x_1(a,b))).
\end{equation}
This allows us to reformulate problem \eqref{eq:EDdeg_problem} as $\min_{a,b \in \RR} f(a,b)$, where
\begin{equation}
    f(a,b) = \|x_1(a,b) - u_1\|^2 + \sum_{j=2}^m \|H_{j,1}x_1(a,b) - u_j\|^2. 
\end{equation}
If $a^*,b^*$ satisfy $(a,b) = \arg\min_{a,b \in \RR} f(a,b)$, then there holds
\begin{equation} \label{eq:derivative_constraint}
    \frac{\partial f(a,b)}{\partial a}(a^*,b^*) = \frac{\partial f(a,b)}{\partial b}(a^*,b^*) = 0.
\end{equation}
These equations are two independent constraints on the variables $a,b$. Although they are, in general, rational functions, we can transform them into polynomials by multiplying them by their common factors. These polynomial equations for~$m=2$ and $m=3$ are shown in the SM.

We consider two approaches to solve these equations: homotopy continuation implementation MiNuS \cite{fabbri2020minus} (HC) and a Gr\"obner basis automatic generator \cite{DBLP:conf/cvpr/LarssonAO17} (GB). To achieve smaller elimination templates in the GB solvers, we generate them with the original equations \eqref{eq:derivative_constraint} and with additional equations derived from~\eqref{eq:derivative_constraint}. These additional equations are given in the SM. All generated solvers have the predicted number of solutions.
Here, we use the shorthand $2$v and $3$v to denote the cases with $m=2$ and $m=3$ views, respectively.
The runtimes are 172.9 $\mu s$ for 2v HC, 6.9 $\mu s$ for 2v GB, 1341.6 $\mu s$ for 3v HC, and~\mbox{82.9 $\mu s$} for~3v~GB.

To evaluate the stability of the proposed solvers, we sampled random cameras~$C_1,\ldots,C_m$ and random 2D points $u_1,...,u_m$, and a plane $\Pi$, and passed these into the solvers. In the perfect case, the returned solutions $a^*,b^*$ would have the derivative \eqref{eq:derivative_constraint} equal to zero. In reality, this number will be nonzero, and its magnitude will reflect the stability of the solver. Therefore, we calculate the derivatives at each solution $a^*,b^*$ \eqref{eq:derivative_constraint}, and report the \textit{derivative error}~$\epsilon_{d} = \max\left( abs\left(\frac{\partial f(a,b)}{\partial a}\right), abs\left(\frac{\partial f(a,b)}{\partial b}\right) \right)$ in Figure~\ref{fig:stability_tests}. \begin{figure}[H]
    \centering
    \begin{tikzpicture}

\begin{axis}[%
xmode=log,
width=0.4\columnwidth,
height=0.25\columnwidth,
at={(0in,0.389in)},
scale only axis,
xmin=1e-20,
xmax=1000000,
xlabel style={font=\color{white!15!black}},
xlabel={$\epsilon_d$},
xtick={1e-15,1e-10,1e-5,1e-0},
ymin=0,
ymax=15000,
ymode=normal,
yminorticks=true,
axis lines = left,
axis background/.style={fill=white},
title style={font=\bfseries},
title={Derivative error},
ylabel={Frequency},
legend style={at={(0.5,1.3)}, anchor=south, legend cell align=left, align=left, draw=white!15!black, font=\normalsize,nodes={scale=0.9, transform shape}},
legend columns=5
]

%%%%%%%%%%%%%%%%%%%%%%%%%%%%%%%%%%%%%%%%%%%%%%%%%%%%%%%%%%%%%%%%%

%%%%%%%%%%%%%%%%%%%%%%%%%%%%%%%%%%%%%%%%%%%%%%%%%%%%%%%%%%%%%%%%%
\addplot [color=red,line width=1.6pt, mark options={solid, red}]
  table[row sep=crcr]{%
1.28825e-20 28 \\
2.13796e-20 0 \\
3.54813e-20 1 \\
5.88844e-20 2 \\
9.77237e-20 1 \\
1.62181e-19 2 \\
2.69153e-19 1 \\
4.46684e-19 4 \\
7.4131e-19 8 \\
1.23027e-18 1 \\
2.04174e-18 15 \\
3.38844e-18 21 \\
5.62341e-18 40 \\
9.33254e-18 24 \\
1.54882e-17 77 \\
2.5704e-17 136 \\
4.2658e-17 110 \\
7.07946e-17 320 \\
1.1749e-16 406 \\
1.94984e-16 689 \\
3.23594e-16 740 \\
5.37032e-16 1276 \\
8.91251e-16 1404 \\
1.47911e-15 1764 \\
2.45471e-15 2016 \\
4.0738e-15 2317 \\
6.76083e-15 2584 \\
1.12202e-14 2947 \\
1.86209e-14 3172 \\
3.0903e-14 3675 \\
5.12861e-14 4124 \\
8.51138e-14 4594 \\
1.41254e-13 5285 \\
2.34423e-13 5889 \\
3.89045e-13 6550 \\
6.45654e-13 6973 \\
1.07152e-12 7452 \\
1.77828e-12 8033 \\
2.95121e-12 8361 \\
4.89779e-12 8710 \\
8.12831e-12 8728 \\
1.34896e-11 8649 \\
2.23872e-11 8670 \\
3.71535e-11 8519 \\
6.16595e-11 8126 \\
1.02329e-10 7943 \\
1.69824e-10 7594 \\
2.81838e-10 7016 \\
4.67735e-10 6642 \\
7.76247e-10 6170 \\
1.28825e-09 5693 \\
2.13796e-09 5436 \\
3.54813e-09 4881 \\
5.88844e-09 4408 \\
9.77237e-09 4085 \\
1.62181e-08 3552 \\
2.69153e-08 3159 \\
4.46684e-08 2933 \\
7.4131e-08 2535 \\
1.23027e-07 2301 \\
2.04174e-07 1929 \\
3.38844e-07 1638 \\
5.62341e-07 1456 \\
9.33254e-07 1250 \\
1.54882e-06 1014 \\
2.5704e-06 897 \\
4.2658e-06 709 \\
7.07946e-06 639 \\
1.1749e-05 501 \\
1.94984e-05 440 \\
3.23594e-05 326 \\
5.37032e-05 287 \\
8.91251e-05 239 \\
0.000147911 222 \\
0.000245471 159 \\
0.00040738 144 \\
0.000676083 112 \\
0.00112202 90 \\
0.00186209 80 \\
0.0030903 73 \\
0.00512861 63 \\
0.00851138 45 \\
0.0141254 47 \\
0.0234423 33 \\
0.0389045 28 \\
0.0645654 27 \\
0.107152 16 \\
0.177828 10 \\
0.295121 19 \\
0.489779 20 \\
0.812831 14 \\
1.34896 5 \\
2.23872 12 \\
3.71535 11 \\
6.16595 9 \\
10.2329 8 \\
16.9824 1 \\
28.1838 5 \\
46.7735 2 \\
77.6247 26 \\
};
\addlegendentry{2v HC}

\addplot [color=green,line width=1.6pt, mark options={solid, red}]
  table[row sep=crcr]{%
1.28825e-20 31 \\
2.13796e-20 0 \\
3.54813e-20 0 \\
5.88844e-20 1 \\
9.77237e-20 1 \\
1.62181e-19 2 \\
2.69153e-19 4 \\
4.46684e-19 6 \\
7.4131e-19 21 \\
1.23027e-18 8 \\
2.04174e-18 22 \\
3.38844e-18 36 \\
5.62341e-18 74 \\
9.33254e-18 85 \\
1.54882e-17 167 \\
2.5704e-17 257 \\
4.2658e-17 323 \\
7.07946e-17 670 \\
1.1749e-16 877 \\
1.94984e-16 1529 \\
3.23594e-16 1967 \\
5.37032e-16 3043 \\
8.91251e-16 4184 \\
1.47911e-15 5864 \\
2.45471e-15 7284 \\
4.0738e-15 8978 \\
6.76083e-15 10727 \\
1.12202e-14 11766 \\
1.86209e-14 12461 \\
3.0903e-14 12303 \\
5.12861e-14 11974 \\
8.51138e-14 11689 \\
1.41254e-13 11308 \\
2.34423e-13 10571 \\
3.89045e-13 9963 \\
6.45654e-13 9174 \\
1.07152e-12 8296 \\
1.77828e-12 7540 \\
2.95121e-12 6701 \\
4.89779e-12 6098 \\
8.12831e-12 5477 \\
1.34896e-11 4853 \\
2.23872e-11 4323 \\
3.71535e-11 3735 \\
6.16595e-11 3265 \\
1.02329e-10 2924 \\
1.69824e-10 2592 \\
2.81838e-10 2108 \\
4.67735e-10 1815 \\
7.76247e-10 1671 \\
1.28825e-09 1409 \\
2.13796e-09 1219 \\
3.54813e-09 1142 \\
5.88844e-09 879 \\
9.77237e-09 791 \\
1.62181e-08 697 \\
2.69153e-08 584 \\
4.46684e-08 562 \\
7.4131e-08 449 \\
1.23027e-07 374 \\
2.04174e-07 314 \\
3.38844e-07 284 \\
5.62341e-07 231 \\
9.33254e-07 217 \\
1.54882e-06 176 \\
2.5704e-06 129 \\
4.2658e-06 116 \\
7.07946e-06 109 \\
1.1749e-05 94 \\
1.94984e-05 86 \\
3.23594e-05 73 \\
5.37032e-05 63 \\
8.91251e-05 48 \\
0.000147911 48 \\
0.000245471 43 \\
0.00040738 47 \\
0.000676083 27 \\
0.00112202 21 \\
0.00186209 29 \\
0.0030903 11 \\
0.00512861 15 \\
0.00851138 15 \\
0.0141254 14 \\
0.0234423 12 \\
0.0389045 5 \\
0.0645654 3 \\
0.107152 6 \\
0.177828 5 \\
0.295121 3 \\
0.489779 5 \\
0.812831 5 \\
1.34896 3 \\
2.23872 2 \\
3.71535 2 \\
6.16595 3 \\
10.2329 1 \\
16.9824 1 \\
28.1838 1 \\
46.7735 3 \\
77.6247 8 \\
};
\addlegendentry{2v GB}

\addplot [color=blue,line width=1.6pt, mark options={solid, red}]
  table[row sep=crcr]{%
1.28825e-20 35 \\
2.13796e-20 0 \\
3.54813e-20 1 \\
5.88844e-20 3 \\
9.77237e-20 0 \\
1.62181e-19 0 \\
2.69153e-19 2 \\
4.46684e-19 5 \\
7.4131e-19 6 \\
1.23027e-18 5 \\
2.04174e-18 13 \\
3.38844e-18 20 \\
5.62341e-18 35 \\
9.33254e-18 37 \\
1.54882e-17 62 \\
2.5704e-17 113 \\
4.2658e-17 80 \\
7.07946e-17 272 \\
1.1749e-16 368 \\
1.94984e-16 639 \\
3.23594e-16 680 \\
5.37032e-16 1200 \\
8.91251e-16 1468 \\
1.47911e-15 1931 \\
2.45471e-15 2217 \\
4.0738e-15 2617 \\
6.76083e-15 3214 \\
1.12202e-14 3505 \\
1.86209e-14 3586 \\
3.0903e-14 3734 \\
5.12861e-14 3919 \\
8.51138e-14 4321 \\
1.41254e-13 4520 \\
2.34423e-13 4833 \\
3.89045e-13 5011 \\
6.45654e-13 5549 \\
1.07152e-12 5836 \\
1.77828e-12 6394 \\
2.95121e-12 6946 \\
4.89779e-12 7514 \\
8.12831e-12 8069 \\
1.34896e-11 8831 \\
2.23872e-11 9600 \\
3.71535e-11 10537 \\
6.16595e-11 11279 \\
1.02329e-10 12017 \\
1.69824e-10 12747 \\
2.81838e-10 13548 \\
4.67735e-10 13686 \\
7.76247e-10 14477 \\
1.28825e-09 14929 \\
2.13796e-09 14813 \\
3.54813e-09 15102 \\
5.88844e-09 14905 \\
9.77237e-09 14476 \\
1.62181e-08 14077 \\
2.69153e-08 13535 \\
4.46684e-08 12894 \\
7.4131e-08 11835 \\
1.23027e-07 11265 \\
2.04174e-07 10102 \\
3.38844e-07 9430 \\
5.62341e-07 8279 \\
9.33254e-07 7432 \\
1.54882e-06 6610 \\
2.5704e-06 5848 \\
4.2658e-06 4992 \\
7.07946e-06 4264 \\
1.1749e-05 3818 \\
1.94984e-05 3271 \\
3.23594e-05 2980 \\
5.37032e-05 2369 \\
8.91251e-05 2041 \\
0.000147911 1725 \\
0.000245471 1617 \\
0.00040738 1227 \\
0.000676083 1070 \\
0.00112202 886 \\
0.00186209 771 \\
0.0030903 651 \\
0.00512861 557 \\
0.00851138 485 \\
0.0141254 378 \\
0.0234423 348 \\
0.0389045 291 \\
0.0645654 245 \\
0.107152 210 \\
0.177828 173 \\
0.295121 145 \\
0.489779 138 \\
0.812831 127 \\
1.34896 83 \\
2.23872 83 \\
3.71535 62 \\
6.16595 71 \\
10.2329 45 \\
16.9824 43 \\
28.1838 45 \\
46.7735 42 \\
77.6247 487 \\
};
\addlegendentry{3v HC}

\addplot [color=cyan,line width=1.6pt, mark options={solid, red}]
  table[row sep=crcr]{%
1.28875e-16 119 \\
2.14047e-16 57 \\
3.55506e-16 106 \\
5.90455e-16 163 \\
9.80676e-16 227 \\
1.62879e-15 338 \\
2.70522e-15 510 \\
4.49306e-15 728 \\
7.46245e-15 1072 \\
1.23943e-14 1415 \\
2.05854e-14 1791 \\
3.41899e-14 2388 \\
5.67855e-14 2889 \\
9.4314e-14 3576 \\
1.56645e-13 4069 \\
2.60168e-13 4685 \\
4.32109e-13 5289 \\
7.17682e-13 5835 \\
1.19199e-12 6474 \\
1.97975e-12 7038 \\
3.28813e-12 7606 \\
5.4612e-12 7884 \\
9.07041e-12 8558 \\
1.50649e-11 8665 \\
2.5021e-11 8997 \\
4.1557e-11 9161 \\
6.90213e-11 9403 \\
1.14636e-10 9452 \\
1.90397e-10 9429 \\
3.16228e-10 9650 \\
5.25217e-10 9769 \\
8.72324e-10 9554 \\
1.44883e-09 9195 \\
2.40633e-09 9261 \\
3.99664e-09 9274 \\
6.63795e-09 8880 \\
1.10249e-08 8662 \\
1.8311e-08 8412 \\
3.04124e-08 8154 \\
5.05114e-08 7988 \\
8.38936e-08 7750 \\
1.39337e-07 7711 \\
2.31423e-07 7584 \\
3.84367e-07 7136 \\
6.38388e-07 6725 \\
1.06029e-06 6370 \\
1.76101e-06 6312 \\
2.92484e-06 6026 \\
4.85781e-06 5897 \\
8.06826e-06 5568 \\
1.34004e-05 5289 \\
2.22565e-05 5072 \\
3.69655e-05 4874 \\
6.13954e-05 4781 \\
0.000101971 4382 \\
0.000169361 4295 \\
0.000281289 4136 \\
0.000467188 3812 \\
0.000775944 3752 \\
0.00128875 3504 \\
0.00214047 3450 \\
0.00355506 3266 \\
0.00590455 3214 \\
0.00980676 3004 \\
0.0162879 2821 \\
0.0270522 2790 \\
0.0449306 2736 \\
0.0746245 2627 \\
0.123943 2523 \\
0.205854 2526 \\
0.341899 2599 \\
0.567855 2607 \\
0.94314 2730 \\
1.56645 2854 \\
2.60168 2984 \\
4.32109 2976 \\
7.17682 3289 \\
11.9199 3229 \\
19.7975 3086 \\
32.8813 3053 \\
54.612 2899 \\
90.7041 2686 \\
150.649 2497 \\
250.21 2362 \\
415.57 2208 \\
690.213 2024 \\
1146.36 1820 \\
1903.97 1706 \\
3162.28 1554 \\
5252.17 1400 \\
8723.24 1352 \\
14488.3 1195 \\
24063.3 1099 \\
39966.4 1023 \\
66379.5 976 \\
110249 882 \\
183110 816 \\
304124 761 \\
505114 717 \\
838936 650 \\
1.39337e+06 574 \\
2.31423e+06 559 \\
3.84367e+06 515 \\
6.38388e+06 441 \\
1.06029e+07 389 \\
1.76101e+07 326 \\
2.92484e+07 331 \\
4.85781e+07 298 \\
8.06826e+07 267 \\
1.34004e+08 210 \\
2.22565e+08 228 \\
3.69655e+08 230 \\
6.13954e+08 234 \\
1.01971e+09 217 \\
1.69361e+09 181 \\
2.81289e+09 186 \\
4.67188e+09 164 \\
7.75944e+09 4554 \\
};
\addlegendentry{3v GB}

%\legend{}

\end{axis}

\end{tikzpicture}%
    %\vspace{-0.4cm}
    \caption{\textit{Solver stability.} Histogram of the derivative errors $\epsilon_d$, calculated from $10^5$ randomly sampled problems using the solvers from Section~\ref{sec:solvers}. HC stands for homotopy continuation, GB stands for Gr\"obner basis. See Section~\ref{sec:solvers} for details. }
    \label{fig:stability_tests}
\end{figure}
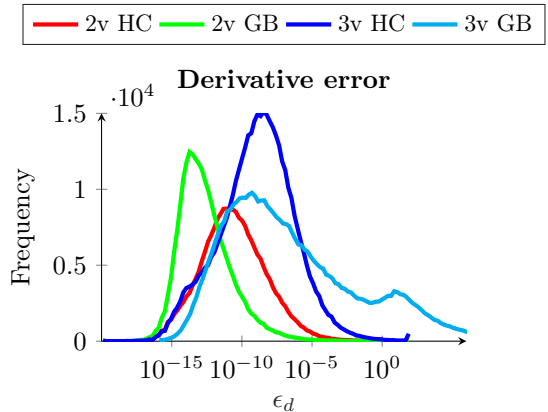

 This experiment demonstrates that both solvers for $2$ views and the HC solver for $3$ views are stable. While the GB solver for $3$ views is also mostly stable, it fails in some cases. Since its runtime is much lower than the runtime of its HC counterpart, we propose to use it first, and if it fails to find a solution with a reasonably low reprojection error, we use the HC solver as a fallback option.
 
\section{Numerical Experiments}
In this section, we present numerical experiments for the triangulation of planar-anchored point configurations. All code and experimental details are available in the SM, to which we refer for further information.

\subsection{Numerical experiments on synthetic data} \label{sec:synthetic}
Our goal is to reconstruct point configurations from point correspondences that are constrained to a single plane across \( m \) views. Below we describe two natural triangulation approaches for this problem, given a camera configuration \( \mathcal{C} \):

-\textbf{(P).UC}: Triangulate each point correspondence independently by fitting it to the unconstrained point multiview variety, that is, by finding the closest point correspondence in the point multiview variety.

-\textbf{(P).C}: Triangulate all point correspondences jointly by enforcing the planar constraint, \ i.e., by fitting the data to the planar-anchored multiview variety~\( \mathcal{M}^\Pi_{\mathcal{C}} \).

While (P).C explicitly incorporates the planar incidence constraint into the triangulation, (P).UC ignores it. As a consequence, the triangulation produced by (P).UC does not preserve the underlying planarity of the reconstructed points, whereas (P).C does.

Our numerical experiments are implemented in \texttt{HomotopyContinuation.jl}~\cite{HC.jl} using \texttt{Julia}~\cite{bezanson2012julia}.

\textbf{Evaluating error.} In each experiment, we first generate a random projective plane $\Pi$ and select $n$ points $X_i \in \Pi$. These points are projected by randomly generated cameras, producing point correspondences in $(\mathbb{R}^2)^m$. For a fixed noise level $\varepsilon = 10^{-12}$, independent noise vectors $\eta_i$ of magnitude $\varepsilon$ are added to each image coordinate.

The triangulation methods then reconstruct points $W_i \in \mathbb{R}^3$, constrained to lie on a plane, that best fit the noisy correspondences.
Reconstruction accuracy is evaluated using the logarithm of the averaged relative triangulation error
\[
\mathcal{E}_{tr}=\log_{10}\!\Bigg(\frac{1}{n\,\varepsilon}\sum_i\|W_i-X_i\|\Bigg),
\]
where $\|W_i-X_i\|$ denotes the Euclidean distance. This quantity measures the amplification of image noise during triangulation.

\textbf{Algorithms.} The pseudocode for the planar-constrained triangulation method is given in Algorithm~\ref{alg:P1}. In the first step, noisy image correspondences are generated by adding randomly sampled noise vectors $\eta$ of magnitude~\mbox{$\varepsilon=10^{-12}$} to the image coordinates. The triangulation step then reconstructs the planar points by solving a closest-point problem on the planar-anchored point multiview variety. This optimization problem \ref{eq:EDdeg_problem} is solved by computing the zeros of a system of polynomial critical equations using the \texttt{solve()} function in \texttt{HomotopyContinuation.jl}~\cite{HC.jl}. Finally, the reconstructed points are compared with the ground-truth points, and the logarithmic average triangulation error is reported.

The unconstrained triangulation method follows the same procedure, except that the reconstruction is performed over $\mathbb{R}^3$ rather than being restricted to a plane.
\begin{algorithm}
\caption{Method (P).C: Planar-constrained triangulation}
\label{alg:P1}
\textbf{Input:} Camera arrangement \( \mathcal{C} = (C_1,\dots,C_m) \), points \( X_1,\dots,X_n \subset \Pi \subset \mathbb{R}^3 \) \\
\textbf{Output:} Logarithm of the average relative error
\begin{algorithmic}[1]
\FOR{$j = 1$ to $m$}
  \FOR{$i = 1$ to $n$}
    \STATE $x_{i,j} \leftarrow C_j(X_i) + \eta(\varepsilon)$
  \ENDFOR
\ENDFOR
\FOR{$i = 1$ to $n$}
  \STATE $W_i \leftarrow
  \arg\min\limits_{X \in \Pi}
  \sum\limits_{j=1}^m \| x_{i,j} - C_j(X) \|^2$
\ENDFOR
\STATE $\mathcal{E}_{tr}\leftarrow \log_{10}\!\left( \frac{1}{n\varepsilon}
\sum\limits_{i=1}^n \|W_i - X_i\| \right)$
\RETURN $\mathcal{E}_{tr}$
\end{algorithmic}
\end{algorithm}

\textbf{Numerical results.} The main results of the numerical experiments are shown in Figure~\ref{fig:plane_histograms_234}.\begin{figure}[H]
    \centering
    
    % -------- Row 1 (Errors) --------
    \begin{subfigure}{0.32\linewidth}
        \centering
        \includegraphics[width=\linewidth]{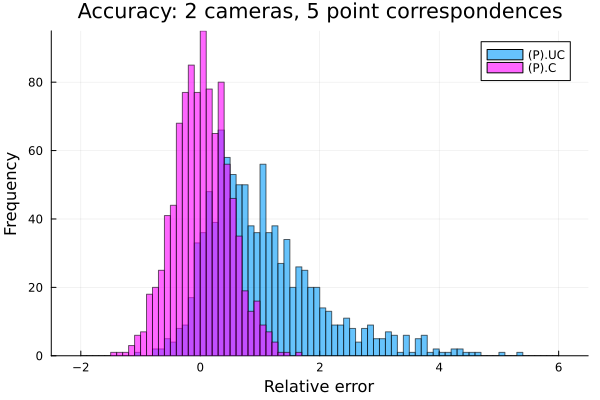}
    \end{subfigure}
    \hfill
    \begin{subfigure}{0.32\linewidth}
        \centering
        \includegraphics[width=\linewidth]{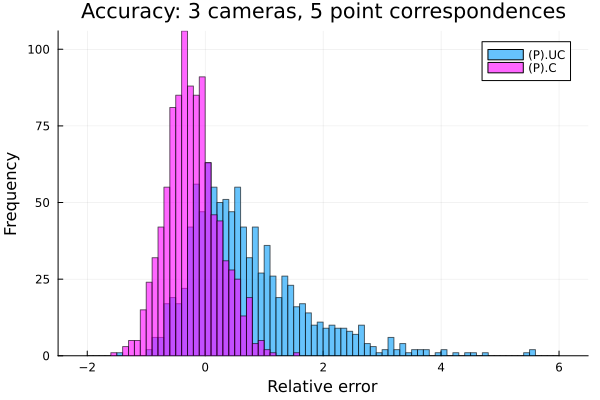}
    \end{subfigure}
    \hfill
    \begin{subfigure}{0.32\linewidth}
        \centering
        \includegraphics[width=\linewidth]{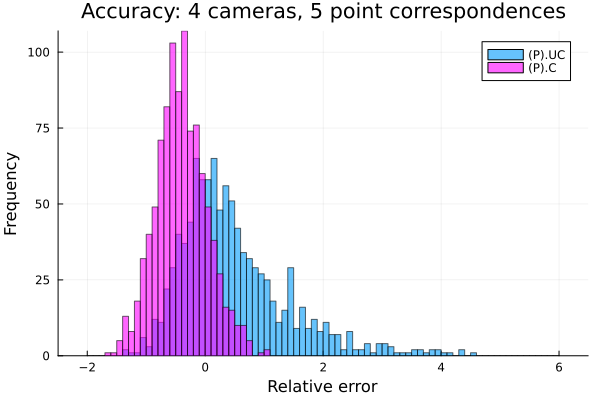}
    \end{subfigure}

    \vspace{0.5em}

    % -------- Row 2 (Runtime) --------
    \begin{subfigure}{0.32\linewidth}
        \centering
        \includegraphics[width=\linewidth]{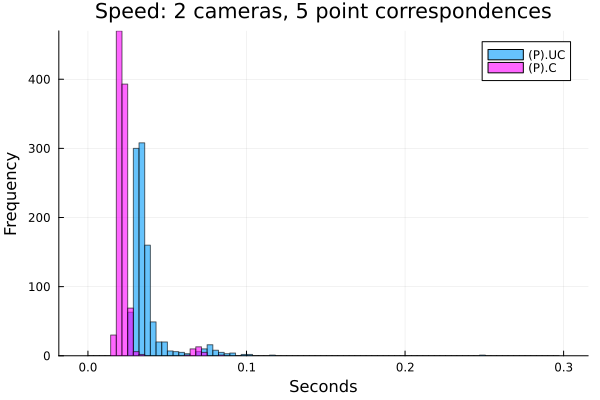}
    \end{subfigure}
    \hfill
    \begin{subfigure}{0.32\linewidth}
        \centering
        \includegraphics[width=\linewidth]{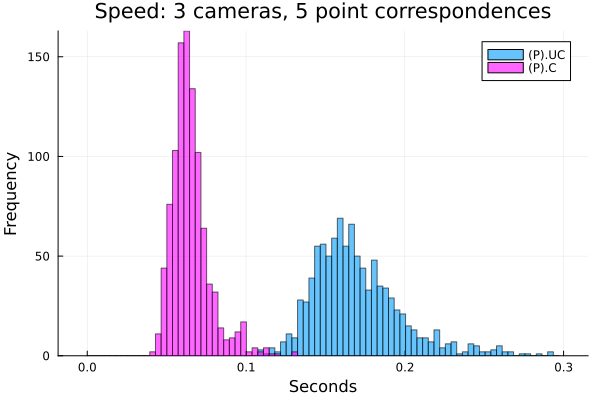}
    \end{subfigure}
    \hfill
    \begin{subfigure}{0.32\linewidth}
        \centering
        \includegraphics[width=\linewidth]{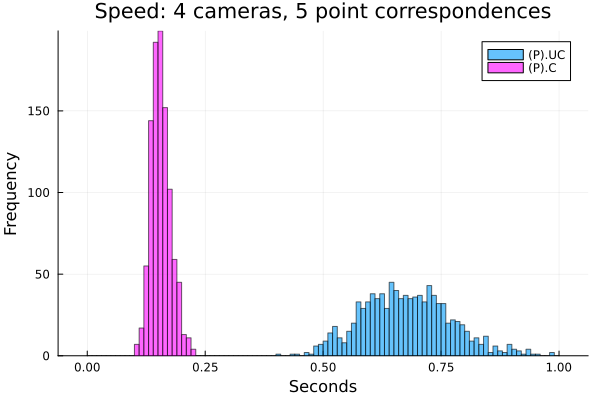}
    \end{subfigure}

    \caption{Triangulation of \( n = 5 \) coplanar correspondences for \( m = 2,3,4 \) cameras over 1000 iterations. Top: logarithmic average relative error. Bottom: runtime distributions.}
    \label{fig:plane_histograms_234}
\end{figure}
We compare the performance of the triangulation methods for $n=5$ and varying numbers of cameras $m$. For each experiment, we record the logarithmic average relative error and the runtime, and repeat each simulation 1000 times. The results are summarized in histograms generated using \texttt{Plots.jl}~\cite{bezanson2012julia}. In our \texttt{HomotopyContinuation.jl} implementation, the planar-constrained method (P).C clearly outperforms the unconstrained approach (P).UC in both accuracy and runtime. Enforcing the planar constraint results in smaller triangulation errors and faster computations.

The promising results observed on synthetic data suggest that explicitly exploiting planar constraints can substantially improve triangulation performance. This observation motivated the development of a specialized solver described in Section~\ref{sec:solvers}, whose numerical performance is evaluated in the following subsection.
\subsection{Numerical experiments on real data}\label{sec:real}
Here, we propose an experiment to evaluate the accuracy of the solvers proposed in Section~\ref{sec:solvers} on real data.
Similarly to Section~\ref{sec:synthetic}, we compare the unconstrained (P).UC, constrained (P).C, and hybrid (P).H strategies:
\begin{itemize}
    \item \textbf{(P).UC}: Triangulate each point correspondence using the standard unanchored two-view \cite{DBLP:conf/caip/HartleyS95} or three-view triangulation solver \cite{DBLP:conf/accv/ByrodJA07}.
    \item \textbf{(P).C}: Triangulate each point correspondence with the Gr\"obner basis solver (Section~\ref{sec:solvers}). If the smallest reprojection error is above $5$ pixels, retriangulate the point with the Homotopy continuation solver. 
    \item \textbf{(P).H}: Apply the (P).C strategy first. If the reprojection error is above $5$ pixels, retriangulate the point with (P).UC.
\end{itemize}
To ensure a good stability of the (P).UC strategy, we implement the solvers \cite{DBLP:conf/caip/HartleyS95,DBLP:conf/accv/ByrodJA07} using homotopy continuation, reflecting that, in this section, we focus on the triangulation accuracy, rather than runtime.

The use of the HC solver in the (P).C strategy reflects the high failure rate of the GB solver. The goal of the (P).H strategy is to filter out points that are incorrectly assigned to the plane despite not lying on it.

We evaluate the solvers on the CAB sequence from the Lamar dataset \cite{sarlin2022lamar}. This is mainly an indoor sequence and contains a large number of planar surfaces. We use the sequence captured with the NavVis M6 trolley, which integrates a RGB camera with a laser scanner, which is used as a source of highly accurate ground truth 3D points in real-world units. We select images recorded by Camera 1, which faces forward and provides sufficient overlap between consecutive frames.

We evaluate our solvers with the following procedure
\begin{itemize}
    \item \textbf{Collecting data}: Take $m$ consecutive images, collect their poses $C_{1},\ldots,C_{m}$, and the point correspondences between them from the COLMAP model associated with the Lamar dataset \cite{sarlin2022lamar}.
    \item \textbf{Plane detection}: Use a monocular depth network \cite{depth_anything_v2} to predict depth, backproject the image points to 3D, and fit planes with RANSAC. We intentionally avoid using the GT laser depth to prevent it from influencing the triangulation process. Although the depth predictions of~\cite{depth_anything_v2} do not produce an accurate 3D geometry, they still provide reliable grouping of points on the same planar surface.
    \item \textbf{Finding plane parameterization $\Pi$}: Select one detected plane, and collect all point correspondences assigned to it, whose Sampson error~\cite{DBLP:books/cu/HZ2004} is below $2px$. Triangulate these points using (P).UC, and fit the plane into the 3D points using RANSAC.
    \item \textbf{Triangulation}: Triangulate all correspondences assigned to the plane using the (P).UC, (P).C, and (P).H strategies.
\end{itemize}
This procedure ensures, that both the plane $\Pi$ and the triangulated points are obtained using only the information available in the images, without using any GT information.

We run this procedure on $m$ tuples of consecutive images from the analyzed sequence, considering planes with at least $10$ correspondences, and measure the \textit{triangulation error} as the Euclidean distance between the triangulated point $X$ and the GT point $X_{GT}$, obtained from the laser scan of the scene:
\begin{equation}
    \epsilon_{tr} = \lVert X - X_{GT} \rVert_2.
\end{equation}

The results, presented in Figure~\ref{fig:real_tests}, demonstrate that, in general, the methods (P).C and (P).H, which use the coplanarity constraints, achieve more accurate triangulation than the unconstrained (P).UC approach. For instance, the main peak for $m=3$ occurs at about $4cm$ for (P).C but at $10cm$ for P.(UC). There is no significant difference between the (P).C and (P).H methods.

The experiments further show, that for $m=2$, all strategies have a non-negligible portion of failed triangulations, which is the highest in the case of (P).UC. This can be observed as the peak on the right side of the histogram. For $m=3$, the portion of failed triangulations is much smaller, but it is again larger for (P).UC than for the other approaches. These experiments demonstrate that the multi-view triangulation constrained to planes is a perspective technique, which has a chance to improve the triangulation accuracy in existing 3D reconstruction pipelines \cite{schoenberger2016sfm,Liu_2023_LIMAP}.

\begin{figure}
    \centering
    \begin{tikzpicture}

\begin{axis}[%
xmode=log,
width=0.38\columnwidth,
height=0.25\columnwidth,
at={(0in,0.389in)},
scale only axis,
xmin=1e-4,
xmax=100,
xlabel style={font=\color{white!15!black}},
xlabel={$\epsilon_{tr}$ (meters)},
xtick={0.001,0.01,0.1,1,10},
title={Triangulation error $m=2$},
ymin=0,
ymax=600,
ymode=normal,
yminorticks=true,
axis lines = left,
axis background/.style={fill=white},
title style={font=\bfseries},
ylabel={Frequency},
legend style={at={(1,1.33)}, anchor=south, legend cell align=left, align=left, draw=white!15!black, font=\normalsize,nodes={scale=0.9, transform shape}},
legend columns=5,
grid=major
]

%%%%%%%%%%%%%%%%%%%%%%%%%%%%%%%%%%%%%%%%%%%%%%%%%%%%%%%%%%%%%%%%%

%%%%%%%%%%%%%%%%%%%%%%%%%%%%%%%%%%%%%%%%%%%%%%%%%%%%%%%%%%%%%%%%%
\addplot [color=red,line width=1.6pt, mark options={solid, red}]
  table[row sep=crcr]{%
0.0001 0\\
0.0001148153621496883 0\\
0.00013182567385564074 0\\
0.00015135612484362088 1\\
0.00017378008287493763 1\\
0.00019952623149688788 3\\
0.00022908676527677723 3\\
0.00026302679918953814 0\\
0.0003019951720402016 1\\
0.0003467368504525317 1\\
0.00039810717055349735 1\\
0.0004570881896148752 1\\
0.0005248074602497723 7\\
0.0006025595860743581 3\\
0.0006918309709189362 4\\
0.0007943282347242813 10\\
0.0009120108393559096 15\\
0.0010471285480508996 8\\
0.001202264434617413 12\\
0.0013803842646028838 15\\
0.001584893192461114 21\\
0.0018197008586099826 33\\
0.0020892961308540386 29\\
0.00239883291901949 38\\
0.002754228703338166 51\\
0.0031622776601683794 45\\
0.003630780547701014 67\\
0.004168693834703355 64\\
0.00478630092322638 86\\
0.005495408738576248 86\\
0.00630957344480193 120\\
0.007244359600749898 98\\
0.008317637711026709 109\\
0.009549925860214359 136\\
0.01096478196143185 137\\
0.012589254117941675 144\\
0.01445439770745928 152\\
0.016595869074375595 189\\
0.019054607179632463 203\\
0.02187761623949552 253\\
0.025118864315095794 261\\
0.028840315031266057 293\\
0.03311311214825911 313\\
0.038018939632056124 338\\
0.04365158322401656 375\\
0.0501187233627272 417\\
0.057543993733715666 397\\
0.06606934480075957 401\\
0.07585775750291836 427\\
0.08709635899560805 416\\
0.1 448\\
0.1148153621496883 491\\
0.13182567385564073 469\\
0.1513561248436207 409\\
0.17378008287493746 443\\
0.1995262314968879 401\\
0.22908676527677724 354\\
0.26302679918953814 371\\
0.3019951720402016 364\\
0.34673685045253166 310\\
0.3981071705534969 318\\
0.4570881896148747 327\\
0.5248074602497723 276\\
0.6025595860743574 253\\
0.6918309709189363 215\\
0.7943282347242814 207\\
0.9120108393559097 184\\
1.0471285480508985 174\\
1.2022644346174132 164\\
1.3803842646028839 140\\
1.584893192461114 107\\
1.8197008586099825 110\\
2.089296130854041 104\\
2.39883291901949 84\\
2.754228703338163 85\\
3.1622776601683795 69\\
3.63078054770101 57\\
4.168693834703355 52\\
4.7863009232263805 39\\
5.4954087385762485 40\\
6.30957344480193 40\\
7.244359600749891 35\\
8.317637711026709 44\\
9.54992586021435 57\\
10.964781961431852 60\\
12.589254117941662 59\\
14.45439770745928 53\\
16.595869074375596 46\\
19.054607179632445 23\\
21.87761623949552 179\\
25.11886431509577 411\\
28.84031503126606 614\\
33.11311214825908 96\\
38.018939632056124 55\\
43.65158322401656 24\\
50.11872336272725 12\\
57.543993733715666 11\\
66.06934480075951 12\\
75.85775750291836 8\\
87.09635899560796 6\\
};
\addlegendentry{(P).UC}

\addplot [color=green,line width=1.6pt, mark options={solid, red}]
  table[row sep=crcr]{%
0.0001 0\\
0.0001148153621496883 0\\
0.00013182567385564074 0\\
0.00015135612484362088 0\\
0.00017378008287493763 0\\
0.00019952623149688788 0\\
0.00022908676527677723 1\\
0.00026302679918953814 0\\
0.0003019951720402016 1\\
0.0003467368504525317 2\\
0.00039810717055349735 2\\
0.0004570881896148752 1\\
0.0005248074602497723 4\\
0.0006025595860743581 5\\
0.0006918309709189362 5\\
0.0007943282347242813 4\\
0.0009120108393559096 16\\
0.0010471285480508996 14\\
0.001202264434617413 10\\
0.0013803842646028838 20\\
0.001584893192461114 26\\
0.0018197008586099826 24\\
0.0020892961308540386 36\\
0.00239883291901949 34\\
0.002754228703338166 50\\
0.0031622776601683794 54\\
0.003630780547701014 93\\
0.004168693834703355 85\\
0.00478630092322638 114\\
0.005495408738576248 105\\
0.00630957344480193 126\\
0.007244359600749898 129\\
0.008317637711026709 146\\
0.009549925860214359 174\\
0.01096478196143185 207\\
0.012589254117941675 215\\
0.01445439770745928 233\\
0.016595869074375595 248\\
0.019054607179632463 256\\
0.02187761623949552 312\\
0.025118864315095794 397\\
0.028840315031266057 375\\
0.03311311214825911 465\\
0.038018939632056124 512\\
0.04365158322401656 462\\
0.0501187233627272 474\\
0.057543993733715666 465\\
0.06606934480075957 480\\
0.07585775750291836 426\\
0.08709635899560805 474\\
0.1 500\\
0.1148153621496883 472\\
0.13182567385564073 512\\
0.1513561248436207 498\\
0.17378008287493746 454\\
0.1995262314968879 384\\
0.22908676527677724 393\\
0.26302679918953814 367\\
0.3019951720402016 332\\
0.34673685045253166 342\\
0.3981071705534969 284\\
0.4570881896148747 275\\
0.5248074602497723 222\\
0.6025595860743574 214\\
0.6918309709189363 173\\
0.7943282347242814 182\\
0.9120108393559097 151\\
1.0471285480508985 135\\
1.2022644346174132 141\\
1.3803842646028839 111\\
1.584893192461114 110\\
1.8197008586099825 81\\
2.089296130854041 89\\
2.39883291901949 73\\
2.754228703338163 63\\
3.1622776601683795 57\\
3.63078054770101 51\\
4.168693834703355 49\\
4.7863009232263805 35\\
5.4954087385762485 37\\
6.30957344480193 29\\
7.244359600749891 43\\
8.317637711026709 44\\
9.54992586021435 42\\
10.964781961431852 47\\
12.589254117941662 42\\
14.45439770745928 53\\
16.595869074375596 45\\
19.054607179632445 19\\
21.87761623949552 24\\
25.11886431509577 100\\
28.84031503126606 19\\
33.11311214825908 18\\
38.018939632056124 13\\
43.65158322401656 24\\
50.11872336272725 10\\
57.543993733715666 12\\
66.06934480075951 15\\
75.85775750291836 10\\
87.09635899560796 5\\
};
\addlegendentry{(P).C}

\addplot [color=blue,line width=1.6pt,dashed, mark options={solid, red}]
  table[row sep=crcr]{%
0.0001 0\\
0.0001148153621496883 0\\
0.00013182567385564074 0\\
0.00015135612484362088 0\\
0.00017378008287493763 0\\
0.00019952623149688788 1\\
0.00022908676527677723 1\\
0.00026302679918953814 0\\
0.0003019951720402016 1\\
0.0003467368504525317 2\\
0.00039810717055349735 2\\
0.0004570881896148752 1\\
0.0005248074602497723 5\\
0.0006025595860743581 5\\
0.0006918309709189362 6\\
0.0007943282347242813 4\\
0.0009120108393559096 17\\
0.0010471285480508996 14\\
0.001202264434617413 10\\
0.0013803842646028838 20\\
0.001584893192461114 27\\
0.0018197008586099826 27\\
0.0020892961308540386 37\\
0.00239883291901949 37\\
0.002754228703338166 53\\
0.0031622776601683794 55\\
0.003630780547701014 96\\
0.004168693834703355 83\\
0.00478630092322638 115\\
0.005495408738576248 107\\
0.00630957344480193 132\\
0.007244359600749898 134\\
0.008317637711026709 158\\
0.009549925860214359 180\\
0.01096478196143185 215\\
0.012589254117941675 221\\
0.01445439770745928 236\\
0.016595869074375595 255\\
0.019054607179632463 261\\
0.02187761623949552 323\\
0.025118864315095794 404\\
0.028840315031266057 380\\
0.03311311214825911 472\\
0.038018939632056124 516\\
0.04365158322401656 460\\
0.0501187233627272 479\\
0.057543993733715666 470\\
0.06606934480075957 482\\
0.07585775750291836 424\\
0.08709635899560805 476\\
0.1 487\\
0.1148153621496883 476\\
0.13182567385564073 493\\
0.1513561248436207 492\\
0.17378008287493746 442\\
0.1995262314968879 392\\
0.22908676527677724 380\\
0.26302679918953814 371\\
0.3019951720402016 332\\
0.34673685045253166 329\\
0.3981071705534969 273\\
0.4570881896148747 243\\
0.5248074602497723 183\\
0.6025595860743574 194\\
0.6918309709189363 160\\
0.7943282347242814 166\\
0.9120108393559097 150\\
1.0471285480508985 125\\
1.2022644346174132 134\\
1.3803842646028839 104\\
1.584893192461114 95\\
1.8197008586099825 76\\
2.089296130854041 80\\
2.39883291901949 68\\
2.754228703338163 62\\
3.1622776601683795 48\\
3.63078054770101 42\\
4.168693834703355 45\\
4.7863009232263805 27\\
5.4954087385762485 35\\
6.30957344480193 26\\
7.244359600749891 38\\
8.317637711026709 37\\
9.54992586021435 54\\
10.964781961431852 56\\
12.589254117941662 47\\
14.45439770745928 54\\
16.595869074375596 44\\
19.054607179632445 21\\
21.87761623949552 35\\
25.11886431509577 142\\
28.84031503126606 121\\
33.11311214825908 38\\
38.018939632056124 8\\
43.65158322401656 16\\
50.11872336272725 7\\
57.543993733715666 10\\
66.06934480075951 14\\
75.85775750291836 11\\
87.09635899560796 5\\
};
\addlegendentry{(P).H}

%\legend{}

\end{axis}

\input{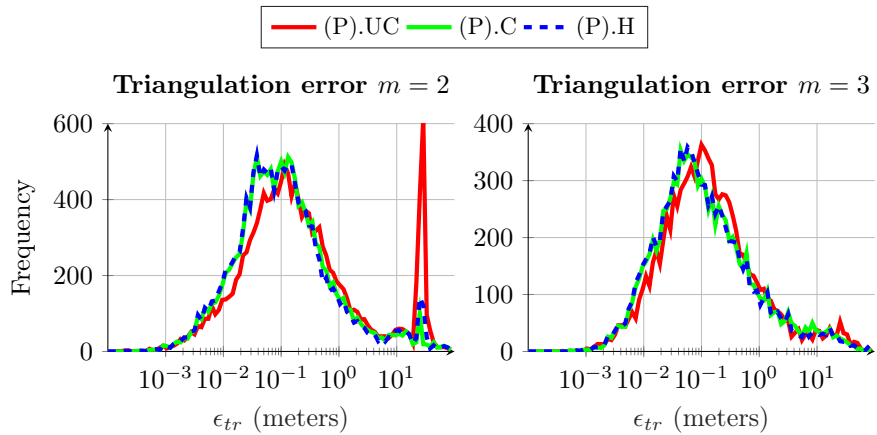}

\end{tikzpicture}%
    \caption{\textit{Real World Experiment.} Histogram of the Euclidean triangulation errors. \textit{left}: $m=2$, \textit{right}: $m=3$. On Lamar \cite{sarlin2022lamar}. In meters. (P).UC: unconstrained solvers \cite{DBLP:conf/caip/HartleyS95,DBLP:conf/accv/ByrodJA07}, (P).C: solvers constrained to plane (Section~\ref{sec:solvers}), (P).H: combinations of both methods. See Section~\ref{sec:real} for details.}
    \label{fig:real_tests}
\end{figure}

\section{Conclusions}
In this paper, we have studied the problem of multi-view triangulation constrained to a plane. We have studied the properties of the planar-anchored point multiview varieties $\mathcal{M}_{\mathcal{C}}^{\Pi}$, in particular we prove that its ED degree equals $\frac{9}{2}m^2 - \frac{13}{2}m + 3$ for every number of views $m$. We have then built minimal solvers for the cases~$m=2$ and $m=3$, and demonstrated, on synthetic and real data, that, on planar scenes, our approach achieves a higher accuracy than the classical unconstrained triangulation.

In future work, this study could be extended to multi-view triangulation of lines, and the method could be used within 3D reconstruction pipelines, such as COLMAP~\cite{schoenberger2016sfm} and Limap~\cite{Liu_2023_LIMAP}. Furthermore, as the approximations \cite{DBLP:conf/3dim/KukelovaPB13,DBLP:conf/cvpr/HedborgRF14} of the optimal unconstrained three-view triangulation \cite{DBLP:conf/accv/ByrodJA07} achieve a superior runtime with a comparable stability, another possible future direction could investigate the possibilities of similar approximations for the planar-anchored triangulation. The code will be made publicly available.

\printbibliography

\newpage
\appendix

\begin{center}
{\Large \textbf{Triangulation of Points Constrained to a Plane}}\\[2ex]
{\large Supplementary Material\\[2ex]
}
\vspace*{2ex}
\end{center}

\noindent
In this appendix, we provide the proof of the main theorem establishing the formula for the ED degree of the planar-anchored point multiview variety. We also give further details on the experiments performed on real data.
\section{Planar-Anchored Point Multiview Varieties}
Let $\Pi\subset\PP^3$ be a plane and let
$\mathcal C=(C_1,\dots,C_m)$ be a camera arrangement with no center on
$\Pi$.  The planar-anchored point multiview variety
\[
\mathcal M^\Pi_{\mathcal C}\subset(\PP^2)^m
\]
is the Zariski closure of the image of the map
\[
\Phi^\Pi_{\mathcal C}:\Pi\to(\PP^2)^m,
\qquad
X\mapsto (C_1X,\dots,C_mX).
\]
By Proposition~\ref{bireg.isomorphism}, we have $\mathcal M^\Pi_{\mathcal C}\cong\PP^2$.
\subsection{Preliminaries}

In this section we recall several notions that will be used in the proof of
Theorem~\ref{EDdeg_plane-thm}.  Throughout, we work over the complex numbers.

\paragraph{Euler characteristic.}
For a complex algebraic variety $V$, we denote by $\chi(V)$ its topological Euler characteristic. The Euler characteristic is a topological invariant that is preserved under homeomorphisms, and in particular under isomorphisms of algebraic varieties \cite{hatcher2005algebraic}. 
It satisfies the additivity property: if $U \subset V$ is a closed subvariety, then
\[
\chi(V) = \chi(U) + \chi(V \setminus U).
\]
More generally, for a finite union $V=\bigcup_i V_i$, the inclusion--exclusion principle yields
\[
\chi\!\left(\bigcup_i V_i\right)
=
\sum_i \chi(V_i)
-\sum_{i<j}\chi(V_i\cap V_j)
+\cdots .
\]

\begin{lemma}\cite{may1999concise}
The Euler characteristic of projective space satisfies
\[
\chi(\PP^n)=n+1 .
\]
\end{lemma}
\paragraph{Linear systems of divisors.}

We briefly recall linear systems of divisors and the basepoint-free property, which will be used in the application of Bertini’s theorem. 
Let $V$ be a smooth projective variety and $D$ a divisor on $V$. 
The complete linear system associated with $D$ is
\[
|D| := \mathbb{P}\bigl(H^0(V,\mathcal O_V(D))\bigr),
\]
whose points correspond to effective divisors linearly equivalent to $D$. A point~\mbox{$x\in V$} is a \emph{base point} if every divisor in the system contains $x$; the set of such points is the \emph{base locus}. 
The system is \emph{basepoint-free} if this locus is empty, equivalently if the global sections of $\mathcal O_V(D)$ generate the line bundle at every point of $V$. 

We will use the following consequence of Bertini’s theorem.

\begin{theorem}[Smooth divisors and Bertini's theorem.]
Let $V$ be a smooth complex variety and let $\Gamma$ be a positive-dimensional basepoint-free linear system on $V$. 
Then a general divisor $D' \in \Gamma$ is smooth.
\end{theorem}

\paragraph{Whitney stratification and correction of the Euler characteristic.}

A Whitney stratification decomposes a variety into smooth locally closed strata satisfying certain regularity conditions. 
To study divisors along the strata one uses Milnor fibers. 
If $f$ is a local defining equation of a divisor $V$ near a point $x$, the Milnor fiber is
\[
F_x = B_{\varepsilon,x}\cap\{f=t\},
\]
for sufficiently small $\varepsilon$ and $t\neq 0$. 
The Euler characteristic of $F_x$ is independent of the choice of $f$ and constant along the stratum containing $x$.

\begin{theorem}\cite{ParusinskiPragacz1995}\cite[Theorem 10.4.4]{MaximIntersection} 
\label{thm: correction of euler charachteristic of singular divisor}
Let $V$ be a smooth complex projective variety and let $U\subset V$ be a divisor admitting a Whitney stratification
\[
U=\bigsqcup_{S\in\mathscr S} S .
\]
Let $W\subset V$ be a smooth divisor linearly equivalent to $U$ meeting $U$ transversely with respect to this stratification. 
Then
\[\label{eq:C17}      
\chi(W)-\chi(U)
=
\sum_{S\in\mathscr S}\mu_S\,\chi(S\setminus W),
\]
where $\mu_S$ denotes the Euler characteristic of the reduced cohomology of the Milnor fiber at a point of the stratum $S$.
\end{theorem}

\subsection{Proof of the Main Theorem: ED Degree of the Planar-Anchored Point Multiview Variety}

We now prove Theorem~\ref{EDdeg_plane-thm} from the main body of the paper.

\begin{maintheorem}[Theorem~\ref{EDdeg_plane-thm}]\label{maintheorem}
Let $\mathcal{C}$ be a generic arrangement of $m \ge 2$ cameras. Then the Euclidean distance degree of the affine planar-anchored point multiview variety is
\[
\operatorname{EDdeg}\!\left(M_\mathcal{C}^{\Pi}\right)
=
\frac{9}{2}m^{2}
-
\frac{13}{2}m
+
3 .
\]
\end{maintheorem}

\paragraph{Affine charts of $(\PP^2)^m$.}

Each projective plane decomposes as
\[
\PP^2 = \CC^2 \cup \PP^1_\infty,
\]
where $\CC^2$ is the standard affine chart and $\PP^1_\infty$ denotes the line at infinity. 
Thus the product $(\PP^2)^m$ contains the affine chart $\CC^{2m}$, whose complement is
\[
H_\infty = \bigcup_{i=1}^{m} H_{\infty,i},
\]
where $H_{\infty,i}$ is obtained by placing $\PP^1_\infty$ in the $i$-th factor. 
Geometrically, these hypersurfaces correspond to configurations where the $i$-th image point lies at infinity.

\paragraph{Distance hypersurface.}

Let $\beta \in \CC^{2m+1}$ be general parameters. The squared Euclidean distance function on $\CC^{2m}$ is
\[
\sum_{i=1}^{2m} (z_i - \beta_i)^2 + \beta_0 .
\]
Its projective closure in $(\PP^2)^m$ defines a hypersurface $H_Q$. 
Intersecting a subvariety of $(\PP^2)^m$ with $H_Q$ imposes the critical equations of the Euclidean distance function.

\paragraph{Divisors associated with the ED problem.}

Let $\mathcal{M}^\Pi_{\mathcal C}$ denote the projective planar-anchored point multiview variety and let $M^\Pi_{\mathcal C}$ be its affine part. 
We consider the divisors
\[
D_Q^\Pi := \mathcal{M}^\Pi_{\mathcal C} \cap H_Q, \qquad
D_{\infty,i}^\Pi := \mathcal{M}^\Pi_{\mathcal C} \cap H_{\infty,i}, \qquad
D_\infty^\Pi := \mathcal{M}^\Pi_{\mathcal C} \cap H_\infty .
\]
Since $H_\infty$ is the complement of the affine chart $\CC^{2m}$ in $(\PP^2)^m$, the divisor $D_\infty^\Pi$ coincides with the complement of $M^\Pi_{\mathcal C}$ inside $\mathcal{M}^\Pi_{\mathcal C}$. 
The divisor $D_Q^\Pi$ is the closure of the critical locus of the squared distance function.

Our computation of the Euclidean distance degree relies on the following theorem from \cite{Maxim2020}, which relates the Euclidean distance degree to Euler characteristics.

\begin{theorem} \label{theorem:maxim}
Let $V \subset \mathbb{C}^n$ be a smooth variety and let 
$U_\beta$ denote the complement of the hypersurface 
\[
\sum_{i=1}^n (z_i - \beta_i) + \beta_0 = 0
\quad \subset \mathbb{C}^n,
\]
where $z \in \mathbb{C}^n$ and $\beta \in \mathbb{C}^{n+1}$. Then,
\[
\mathrm{EDdeg}(V) = (-1)^{\dim V}\,\chi\bigl(V \cap U_\beta\bigr).
\]
\end{theorem}

Therefore, in order to find 
$ \mathrm{EDdeg} (M^\Pi_{\mathcal{C}}),$ 
we need to compute the Euler characteristic 
$ \chi\!\left(M^\Pi_{\mathcal{C}} \cap U_\beta\right).$ As derived in \cite{Maxim2020}, using additivity and inclusion--exclusion principle we have
\[\label{eq:EDchi}
\chi(M^\Pi_{\mathcal{C}} \cap U_\beta) 
= \chi(\mathcal{M}^\Pi_{\mathcal{C}}) - \chi(D_\infty^\Pi) + \chi(D_Q^\Pi \cap D_\infty^\Pi) - \chi(D_Q^\Pi).\tag{1}
\]

The structure of the proof of Theorem~`\ref{EDdeg_plane-thm} is to calculate the four terms of the above equation. In the following lemma, we compute the first term:
\begin{lemma}\label{eu: plane-anchor MV}
 For a fixed plane $\Pi \subset \mathbb{P}^3$ and a generic arrangement of $m$ cameras~$\mathcal{C}$,
\[
\chi\!\left(\mathcal{M}^{\Pi}_{\mathcal{C}}\right) = 3.
\]
\end{lemma}

\begin{proof}
From Proposition~\ref{bireg.isomorphism}, $\mathcal{M}^{\Pi}_{\mathcal{C}} \simeq \mathbb{P}^2$, hence $\chi(\mathcal{M}^{\Pi}_{\mathcal{C}})=\chi(\mathbb{P}^2)=3$. 
\end{proof}
In the next step, we compute the second terms of the right-hand sides of \eqref{eq:EDchi}.
\begin{lemma}[Divisor at infinity for the anchored points on a plane]
\label{lem:Dinfty-plane}
Fix a plane~$\Pi\subset \PP^{3}$ and let $C=(C_{1},\dots,C_{m})$ be a generic arrangement of $m$ cameras with no center on $\Pi$. Then
$$\chi\!\left(D_{\infty}^{\Pi}\right)=2m-\binom{m}{2}$$.
\end{lemma}

\begin{proof}
Since no camera center lies on $\Pi$, each restriction
$C_{i}\bigr|_{\Pi}:\Pi\to\PP^{2}$ is a projective isomorphism.
Hence the pullback of the line at infinity in the $i$-th factor is a line on
$\Pi\cong\PP^{2}$. Because $\mathcal{M}^{\Pi}_{\mathcal{C}}$ is the image of $\Pi$ under the joint
map, we get that $D_{\infty,i}^{\Pi}$ is (the image of) a line in $\Pi$.
Therefore $D_{\infty,i}^{\Pi}\cong\PP^{1}$ and
\[
\chi\!\left(D_{\infty,i}^{\Pi}\right)=\chi(\PP^{1})=2.
\]

\smallskip
For $i\neq j$, the divisors $D_{\infty,i}^{\Pi}$ and $D_{\infty,j}^{\Pi}$
are distinct lines on the plane~\mbox{$\mathcal{M}^{\Pi}_{\mathcal{C}}\cong\PP^{2}$.}
By genericity, they meet transversely in a single point and no three of the
$D_{\infty,i}^{\Pi}$ are concurrent. So, $\chi\!\left(D_{\infty,i}^{\Pi}\right)$ only have pairwise intersections, and two back-projected lines $A_i, A_j$ through $c_i$ and $c_j$ respectively, meet in exactly a generic point in $\Pi.$ Therefore $\left(D_{\infty,i}^{\Pi}\cap D_{\infty,j}^{\Pi}\right)$ consists of a single element. By additivity of the Euler characteristic and
inclusion--exclusion,
\[
\chi\!\left(D_{\infty}^{\Pi}\right)
=\chi\!\left(\bigcup_{i=1}^{m} D_{\infty,i}^{\Pi}\right)
=\sum_{i=1}^{m}\chi\!\left(D_{\infty,i}^{\Pi}\right)
-\sum_{1\le i<j\le m}\chi\!\left(D_{\infty,i}^{\Pi}\cap D_{\infty,j}^{\Pi}\right)
= \sum_{i=1}^{m} 2 \;-\; \sum_{i<j} 1.
\]
Hence
\[
\chi\!\left(D_{\infty}^{\Pi}\right)
=2m-\binom{m}{2}.
\]
\end{proof}
Recall that $H_Q$ is defined as the closure of the affine hypersurface
\[\label{affin-euclid-hypersurf}
\sum_{i=1}^{2m} (z_i - \beta_i)^2 + \beta_0 = 0
\tag{2}
\]
in $(\mathbb{P}^2)^m$.  
Introduce homogeneous coordinates $(x_i, y_{2i-1}, y_{2i})$ with 
\[
z_{2i-1} = \frac{y_{2i-1}}{x_i}, 
\qquad 
z_{2i}   = \frac{y_{2i}}{x_i}, 
\quad 1 \leq i \leq m.
\]
Let $\mathbf{x} = x_1 \cdots x_m$.  
Then the homogenization of (\ref{affin-euclid-hypersurf}), which gives the defining equation of $H_Q$, is
\[
\sum_{i=1}^m 
\left[ (y_{2i-1} - \beta_{2i-1} x_i)^2 + (y_{2i} - \beta_{2i} x_i)^2 \right] 
\frac{\mathbf{x}^2}{x_i^2} 
\;+\; \beta_0 \mathbf{x}^2 = 0.
\tag{3}
\]

\begin{lemma}[Intersection of $D_{Q}^{\Pi}$ with $D_{\infty}^{\Pi}$]
\label{lem:DQ-infty-plane}
Let $\mathcal{M}^\Pi_{\mathcal{C}}$ be the planar-anchored point multiview variety 
for $m$ cameras in generic position with anchor plane $\Pi \subset \PP^3$.
Then
\[
\chi\!\left(D_{Q}^{\Pi}\cap D_{\infty}^{\Pi}\right) \;=\; 2m + \binom{m}{2}.
\]
\end{lemma}

\begin{proof}
Intersecting the homogenized distance divisor $H_Q$ with the hyperplane at
infinity $H_{\infty,i}=\{x_i=0\}$, one obtains three types of components
as in the general decomposition:
\[
\begin{aligned}
H_Q\cap H_{\infty,i}
&=\{y_{2i-1}+\sqrt{-1}\,y_{2i}=0,\,x_i=0\} \\
&\quad\cup\ \{y_{2i-1}-\sqrt{-1}\,y_{2i}=0,\,x_i=0\} \\
&\quad\cup\ \bigcup_{j\neq i}\{x_i=x_j=0\}.
\end{aligned}
\]

For the planar-anchored point multiview variety $\mathcal{M}^\Pi_{\mathcal C}$, set 
\[
\begin{aligned}
K_{i}^{\Pi,\pm}
&:=\mathcal{M}^\Pi_{\mathcal C}\cap\{y_{2i-1}\pm\sqrt{-1}\,y_{2i}=0,\;x_i=0\},\\
A_{i,j}^{\Pi}
&:=\mathcal{M}^\Pi_{\mathcal C}\cap\{x_i=x_j=0\},
\qquad (i\neq j).
\end{aligned}
\]
\textbf{Contribution of the $K_i^{\Pi,\pm}$ components.}
For a fixed camera $C_i$, the conditions
\[
y_{2i-1}\pm\sqrt{-1}\,y_{2i}=0,\qquad x_i=0
\]
mean that the image point in the $i$-th factor is the isotropic point
\mbox{\([0:\pm \sqrt{-1}:1]\in\PP^2\).}
Back-projecting such a point gives a \emph{complex} line through the camera center. Since the anchor plane $\Pi$ is not parallel to this direction,
this complex line meets~$\Pi$ in a unique point.
Thus each of the two isotropic directions produces a single configuration in $\mathcal{M}^\Pi_{\mathcal C}$.

Hence, for each $i$, the sets $K_i^{\Pi,+}$ and $K_i^{\Pi,-}$ each contribute
one point, giving a total of
\[
\sum_{i=1}^m \chi(K_i^{\Pi,+})+\chi(K_i^{\Pi,-})
=2m.
\]

\textbf{Contribution of the $A_{i,j}^\Pi$ components.}
The components $\{x_i=x_j=0\}$ impose that the image points of cameras $i$ and $j$ both lie at infinity. Geometrically, this means the back-projected rays of
$C_i$ and $C_j$ meet at a point of the anchor plane $\Pi$.
For generic cameras this intersection is unique, so each $A_{i,j}^\Pi$
consists of a single point, and different pairs $\{i,j\}$ give disjoint
configurations. Therefore
\[
\sum_{i<j}\chi(A_{i,j}^\Pi)=\binom{m}{2}.
\]

Adding both contributions gives
\[
\chi(D_Q^\Pi\cap D_\infty^\Pi)
=2m+\binom{m}{2}.
\]
\end{proof}

We recall that we have the following linear equivalence of divisors in $(\mathbb{P}^2)^m$:
\[
H_Q \equiv 2H_{\infty,1} + \cdots + 2H_{\infty,m}.
\]

Then, as divisors of the planar-anchored point multiview variety, we obtain
\[
D_Q^{\Pi} \equiv 2\mathcal{M}_\mathcal{C}^{\Pi} \cap H_{\infty,1} 
+ \cdots + 2\mathcal{M}_\mathcal{C}^{\Pi} \cap H_{\infty,m}.
\]

Consider the projection 
\[
\pi_{\Pi} : \mathcal{M}_\mathcal{C}^{\Pi} \;\longrightarrow\; \Pi,
\]
which sends each tuple of image points to the intersection of their back-projected lines with the anchor plane $\Pi$.  
In the Chow ring of $\mathbb{P}^3$, every point of $\Pi$ is equivalent.  

Let $H$ be a generic point in $\Pi$, and define
\[
D_H^{\Pi} := \pi_{\Pi}^*(H),
\]  
to be the preimage of a generic hyperplane in $\Pi$, i.e. a generic
point of $\Pi.$ 

Recall that a divisor on $\PP^2$ determines a family of all
divisors that are linearly equivalent to it, called its \emph{linear system}.  In our setting,
the divisor $D_Q^\Pi$ has degree~$2m$, so $\mathcal{O}(D_Q^\Pi)$ corresponds to the
line bundle of plane curves of degree~$2m$ on $\PP^2$.  The linear system
$\Gamma(\mathcal{M}^\Pi_{\mathcal C},\mathcal{O}(D_Q^\Pi))$ therefore parametrizes all
plane curves of degree $2m$ that play the same geometric role as $D_Q^\Pi$.

\begin{lemma}
A generic divisor $D^{\Pi'}$ in the linear system 
$\Gamma(\mathcal{M}^\Pi_{\mathcal C}, \mathcal{O}(D_Q^\Pi))$ is smooth.
\end{lemma}

\begin{proof}
On $(\PP^2)^m$, recall that 
\[
H_Q \equiv 2H_{\infty,1} + \cdots + 2H_{\infty,m}.
\]
Restricting to the planar-anchored point multiview variety 
$\mathcal{M}^\Pi_{\mathcal C}\cong \PP^2$, we obtain
\[
D_Q^\Pi \equiv 2m \cdot H,
\]
where $H=[D_H^\Pi]$ is the hyperplane class on $\PP^2$.

Thus the complete linear system of divisors linearly equivalent to $D_Q^\Pi$
is the linear system of plane curves of degree $2m$.
This system is basepoint-free, since no point of $\PP^2$ lies on all 
degree-$2m$ curves.
By Bertini’s theorem, a generic member of this linear system is smooth.
\end{proof}
\begin{proposition}\label{prop: generic divisor}
Let $D^{\Pi'}$ be a generic divisor in the linear system 
$\Gamma(\mathcal{M}^\Pi_{\mathcal C}, \mathcal{O}(D_Q^\Pi))$. Then
\[
\chi(D^{\Pi'}) = -4m^2 + 6m.
\]
\end{proposition}

\begin{proof}
On $\mathcal{M}^\Pi_{\mathcal C}\cong \PP^2$, we have 
$D_Q^\Pi \equiv 2m H$, where $H=[D_H^\Pi]$ is the hyperplane class. Thus $D^{\Pi'}$ is a smooth plane curve of degree $2m$. Hence, the genus $g$ of~$D^{\Pi'}$ is
\[
g = \frac{(d-1)(d-2)}{2}, \quad \text{with } d = \deg(D^{\Pi'}) = 2m.
\]
Therefore
\[
g = \frac{(2m-1)(2m-2)}{2}.
\]

The Euler characteristic of a smooth curve is
\[
\chi(D^{\Pi'}) = 2 - 2g.
\]
Substituting gives
\[
\chi(D^{\Pi'}) = 2 - (2m-1)(2m-2) = -4m^2 + 6m.
\]
\end{proof}
In the upcoming paragraph we will make use of Theorem~\ref{thm: correction of euler charachteristic of singular divisor}. 

For each camera $i=1,\dots,m$, let $L_i:=\mathcal M^\Pi_{\mathcal C}\cap H_{\infty,i}$ and
$p_{ij}:=L_i\cap L_j$.
% === 2) Singular strata of D_Q^\Pi ===

\begin{proposition}[Singular strata]\label{prop:sing-strata}
The actual divisor $D_Q^\Pi$ has singular locus consisting of $\binom{m}{2}$ ordinary nodes.
\end{proposition}

\begin{proof}
\noindent
In the ambient $(\PP^2)^m$ one has the standard decomposition
\[
H_Q\cap H_{\infty,i}\;=\;\{\,y_{2i-1}\pm \mathrm{i}\,y_{2i}=x_i=0\,\}\;\cup\;\bigcup_{j\neq i}\{\,x_i=x_j=0\,\}.
\]
Restricting to $\mathcal M^\Pi_{\mathcal C}\cong\PP^2$:
- the first piece cuts $L_i$ in \emph{two} points (the two solutions of $y_{2i-1}^2+y_{2i}^2=0$), and
- the second piece yields the \emph{pairwise intersection points}~$p_{ij}:=L_i\cap L_j$.

Thus any potential singularity of $D_Q^\Pi$ must lie either at one of the two points on $L_i$ (the ``isotropic points’’) or at a pairwise intersection $p_{ij}$.

\smallskip
\noindent\textit{ Isotropic points are smooth.}
Fix $i$ and one isotropic point on $L_i$. Choose affine coordinates $(u,v)$ near \(p_{ij}\) on $\PP^2$ centered there with
$L_i=\{u=0\}$. The restriction of $H_Q$ to $\PP^2$ has a local defining
equation of the form
\[
f(u,v)=a(u,v)\,u+\text{(terms of order $\ge2$)},\qquad a(0,0)\neq0,
\]
because the isotropic condition in the $i$-th factor imposes one transverse linear equation on $L_i$.
Hence $\partial_u f(0,0)\neq0$ and $\nabla f\neq0$; by the implicit function theorem, $D_Q^\Pi$ is smooth
at each isotropic point.

\smallskip
\noindent\textit{ Local normal form at $p_{ij}=L_i\cap L_j$.}
Fix $i\neq j$ and choose affine coordinates $(u,v)$ centered at $p_{ij}$ with $L_i=\{u=0\}$ and
$L_j=\{v=0\}$. Freezing the other factors, the restriction of $H_Q$ to $\PP^2$ has the Taylor expansion
\[
f(u,v)\;=\;\alpha(u,v)\,u^2\;+\;\beta(u,v)\,v^2\;-\;\gamma(u,v)\,u^2v^2\;+\;\text{(terms of order }\ge3),
\]
where $\alpha(0,0),\beta(0,0),\gamma(0,0)\neq0$. Dividing by a unit and rescaling by local square roots, we obtain
the analytic normal form
\[
f(u,v)\;=\;u^2+v^2-u^2v^2.
\]
Then
\[
\partial_u f=2u(1-v^2),\qquad \partial_v f=2v(1-u^2),\qquad \mathrm{Hess}_{(0,0)}(f)=
\begin{pmatrix} 2 & 0\\ 0 & 2\end{pmatrix},
\]
so $\nabla f(0,0)=0$ and the quadratic part $u^2+v^2$ is nondegenerate. Hence $p_{ij}$ is an $A_1$ singularity,
i.e.\ an \emph{ordinary node}.
\smallskip
\noindent\textit{ No other singularities.}
Away from the finite set $\{p_{ij}\}_{i<j}$ we are either at an isotropic point (already smooth) or
off all $L_i$, where the restriction of $H_Q$ to the $\PP^2$-anchor is transverse (the affine quadratic form
has nonzero gradient). Thus there are no additional singular points.

\smallskip
Counting the pairwise intersections of the $m$ lines $L_i$ gives exactly $\binom{m}{2}$ nodes. Therefore
$\mathrm{Sing}(D_Q^\Pi)=\{p_{ij}\}_{i<j}$ and all singularities are ordinary double points.
\end{proof}

% === 3) Milnor fiber at a node ===
\begin{proposition}[Whitney stratification of $D_Q^\Pi$]
\label{prop:Whitney-plane}
A Whitney stratification of $D_Q^\Pi$ consists of:
\begin{enumerate}
\item the stratum of smooth points $S_{\mathrm{reg}} = D_Q^\Pi \setminus \{p_{ij}\}_{i<j}$;
\item the $0$-dimensional strata $S_{i,j} = \{p_{ij}\} = L_i \cap L_j$, for $1\le i<j\le m$.
\end{enumerate}
\end{proposition}

\begin{proof}
The divisor $D_Q^\Pi \subset \PP^2$ is a hypersurface with isolated singularities.  
By~Prop.~\ref{prop:sing-strata}, these singularities occur precisely at the $\binom{m}{2}$ intersection points~\mbox{$p_{ij}=L_i\cap L_j$}, and each is an ordinary node.  
Thus the decomposition
\[
D_Q^\Pi = S_{\mathrm{reg}} \ \sqcup\ \bigcup_{i<j} S_{i,j}
\]
satisfies Whitney’s conditions: the smooth part is open and dense in $D_Q^\Pi$, and each singular stratum $S_{i,j}$ is a transverse isolated point where the tangent cone is a pair of intersecting planes.  
Hence this is a Whitney stratification.
\end{proof}

\begin{lemma}[Milnor fiber at $p_{ij}$]
\label{lem:Milnor-node}
The Euler characteristic of the reduced cohomology of the Milnor fiber at each singular point
$p_{ij}$ of $D_Q^\Pi$ is $-1$.
\end{lemma}

\begin{proof}
Near $p_{ij}=L_i\cap L_j$, we choose local affine coordinates $(u,v)$ on $\PP^2$ such that
$L_i=\{u=0\}$ and $L_j=\{v=0\}$.  As shown in Proposition~\ref{prop:sing-strata},
the local equation of $D_Q^\Pi$ is, up to analytic equivalence,
\[
f(u,v)=u^2+v^2-u^2v^2=0.
\]
We compute the Milnor fiber of this singularity.

Let $F_\varepsilon=\{f(u,v)=\varepsilon\}\cap B_\delta$ for small $0<|\varepsilon|\ll\delta$.  
Set $x=u^2$ and $y=v^2$; then the defining equation becomes
\[
x+y-xy=\varepsilon,
\]
and we denote $G_t=\{x+y-xy=t\}\cap B_\epsilon$ for small $t$.  
The map
\[
\psi:(u,v)\longmapsto(x,y)=(u^2,v^2)
\]
is $4$-to-$1$ on $\{xy\neq0\}$ and $2$-to-$1$ along $\{x=0\}$ and $\{y=0\}$.  
Hence
\[
\chi(F_\varepsilon)
=4\,\chi(G_t\cap\{xy\neq0\})
 +2\,\chi(G_t\cap\{x=0\})
 +2\,\chi(G_t\cap\{y=0\}).
\]
The gradient of $x+y-xy$ is $(1-y,1-x)$, which does not vanish near the origin,
so~$G_t$ is smooth. Furthermore, \(G_t\cap\{x=0\}\) and \(G_t\cap\{y=0\}\) are single points, while
\(G_t\cap\{xy\neq 0\}\) is a disk with these two points removed, so
\[
\chi(G_t\cap\{xy\neq 0\})=1-2.
\]
Substituting gives
\[
\chi(F_\varepsilon)=4(1-2)+2(1)+2(1)=0.
\]
Since $F_{\varepsilon}$ is homotopy equivalent to $\mathbb C^{*}$, its unreduced Euler characteristic is~$0$, and thus the reduced Euler characteristic is $\tilde\chi(F_\varepsilon)=-1$.
Equivalently,~\mbox{$\widetilde H^1(F_\varepsilon;\mathbb Q)\cong\mathbb Q$} and all other reduced cohomology groups vanish.  
Hence each singular point $p_{ij}$ of $D_Q^\Pi$ has Milnor number $\mu=1$ and reduced Euler characteristic $-1$.
\end{proof}

% === 4) Euler characteristic of D_Q^\Pi ===
\begin{proposition}
\label{prop:chi-DQ}
The Euler characteristic of $D_Q^\Pi$ is
\[
\chi(D_Q^\Pi)
=-4m^2+6m+\binom{m}{2}.
\]
\end{proposition}

\begin{proof}
By Theorem~\ref{thm: correction of euler charachteristic of singular divisor} and Proposition~\ref{prop:Whitney-plane}, applied to the
Whitney stratification~\(
D_Q^\Pi=S_{\mathrm{reg}}\ \sqcup\ \bigsqcup_{i<j}S_{i,j},
\)
we have
\begin{equation}\label{eq:chi-diff}
\chi(D^{\Pi'})-\chi(D_Q^\Pi)
=\mu_0\,\chi(S_{\mathrm{reg}}\setminus D^{\Pi'})
+\sum_{i<j}\mu_{i,j}\,\chi(S_{i,j}\setminus D^{\Pi'}),
\end{equation}
where $\mu_0$ and $\mu_{i,j}$ denote the Euler characteristics of the reduced
cohomology of the Milnor fibers along the corresponding strata.

For the smooth stratum $S_{\mathrm{reg}}$, the divisor $D_Q^\Pi$ is nonsingular,
so $\mu_0=0$.  
By Lemma~\ref{lem:Milnor-node}, each singular stratum $S_{i,j}=\{p_{ij}\}$
is an ordinary node, so its Milnor fiber has reduced Euler characteristic $-1$.
Moreover, Proposition~\ref{prop:sing-strata} shows that there are exactly $\binom{m}{2}$
such points.  Substituting these values into~\eqref{eq:chi-diff} gives
\[
\chi(D^{\Pi'})-\chi(D_Q^\Pi)
=\sum_{i<j}\mu_{i,j}\chi(S_{i,j}\setminus D^{\Pi'})
=-\binom{m}{2}.
\]
Hence
\[
\chi(D_Q^\Pi)=\chi(D^{\Pi'})+\binom{m}{2}.
\]
By Proposition~\ref{prop: generic divisor},$
\chi(D^{\Pi'})=-4m^2+6m,$ 
and substituting yields
\[
\chi(D_Q^\Pi)
=-4m^2+6m+\binom{m}{2}.
\]
\end{proof}

Now we conclude

\begin{proof}[Proof of Thorem~\ref{maintheorem}/Theorem~\ref{EDdeg_plane-thm}]
Let us recall that by the ED--$\chi$ relation and projective additivity, we got \eqref{eq:EDchi}
\[
\mathrm{EDdeg}(M^\Pi_{\mathcal{C}})
=\chi(\mathcal{M}^\Pi_{\mathcal{C}})
-\chi(D_\infty^\Pi)
+\chi(D_Q^\Pi\cap D_\infty^\Pi)
-\chi(D_Q^\Pi).
\]
Each term can be computed as follows.

\smallskip
\noindent
(1) From Lemma~\ref{eu: plane-anchor MV}, the anchored--plane multiview variety is
$\mathcal{M}^\Pi_{\mathcal{C}}\cong\PP^2$, hence
\[
\chi(\mathcal{M}^\Pi_{\mathcal{C}})=3.
\]

\smallskip
\noindent
(2) By Lemma~\ref{lem:Dinfty-plane}, the divisor
$D_\infty^\Pi=\bigcup_{i=1}^m L_i$
is a union of $m$ lines in general position on $\PP^2$, giving
\[
\chi(D_\infty^\Pi)
=\sum_{i=1}^m \chi(\PP^1)-\sum_{i<j}\chi(\mathrm{pt})
=2m-\binom{m}{2}.
\]

\smallskip
\noindent
(3) By Lemma~\ref{lem:DQ-infty-plane}, the intersection
$D_Q^\Pi\cap D_\infty^\Pi$ consists of the $2m$ isotropic points
on the $L_i$’s together with all pairwise intersection points,
so
\[
\chi(D_Q^\Pi\cap D_\infty^\Pi)=2m+\binom{m}{2}.
\]

\smallskip
\noindent
(4) Finally, by Proposition~\ref{prop:chi-DQ},
the divisor $D_Q^\Pi$ has Euler characteristic
\[
\chi(D_Q^\Pi)
=-4m^2+6m+\binom{m}{2}.
\]

\smallskip
Substituting all values into Equation (1) gives
\[
\begin{aligned}
\mathrm{EDdeg}(M^\Pi_{\mathcal{C}})
&=3-\bigl(2m-\tbinom{m}{2}\bigr)
+\bigl(2m+\tbinom{m}{2}\bigr)
-\bigl(-4m^2+6m+\tbinom{m}{2}\bigr)\\
&=3+m(m-1)+4m^2-6m-\tfrac{m(m-1)}{2}\\
&=\tfrac{9}{2}m^2-\tfrac{13}{2}m+3.
\end{aligned}
\]
This completes the proof.
\end{proof}

\section{Experiments on real data}
\begin{figure}
    \centering
    \begin{tabular}{c c c}
       \includegraphics[width=0.25\linewidth]{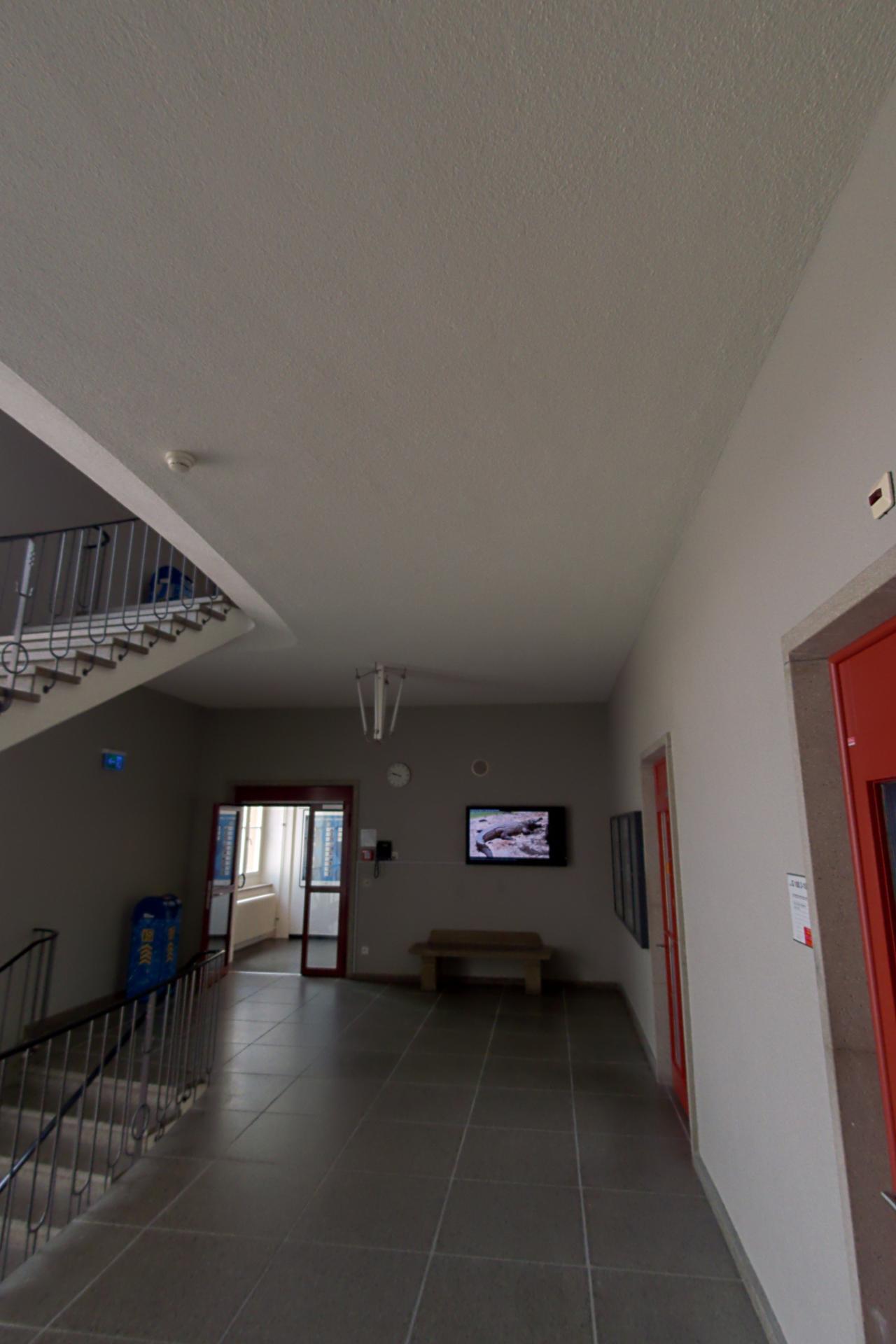} &
       \includegraphics[width=0.25\linewidth]{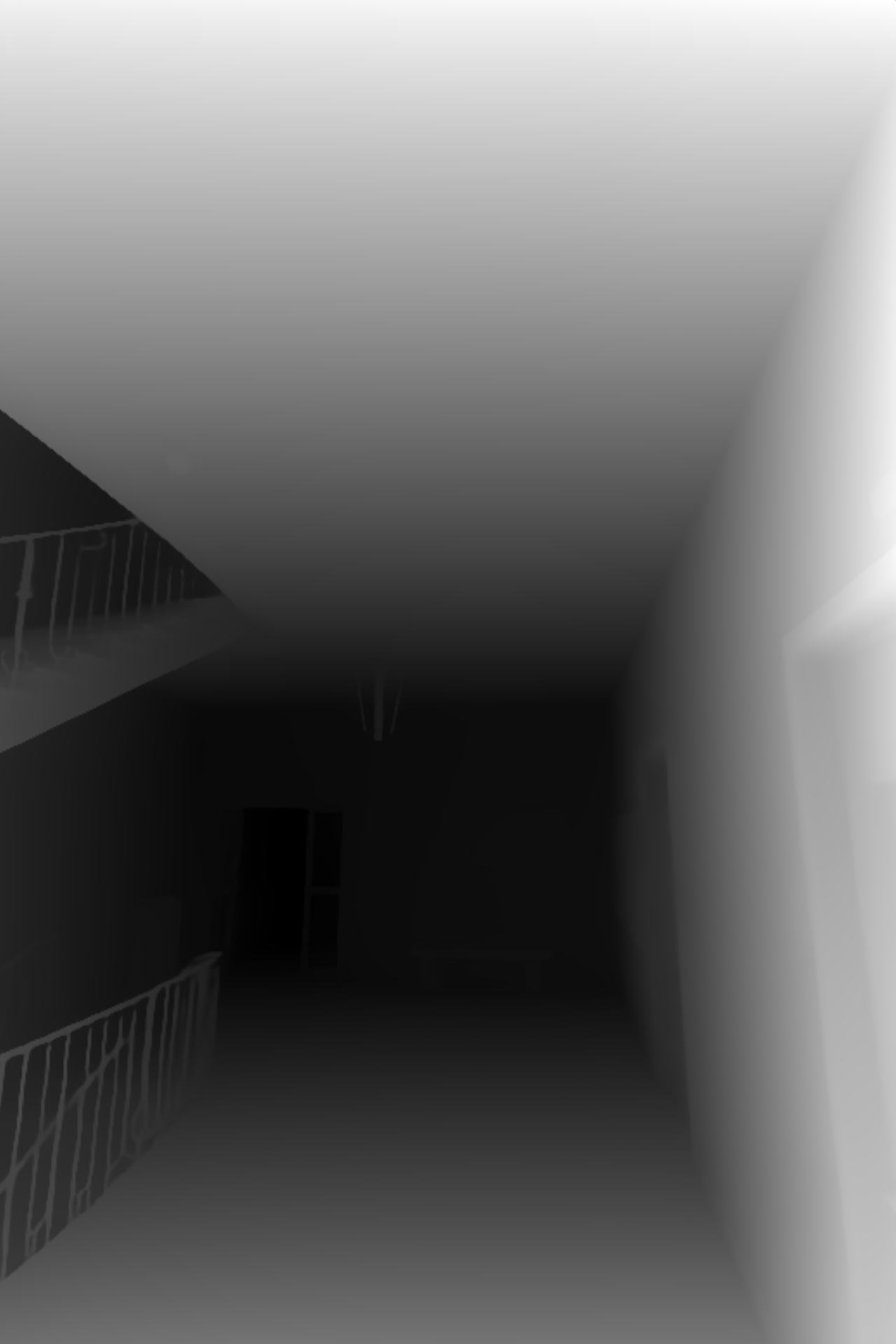} &
       \includegraphics[width=0.25\linewidth]{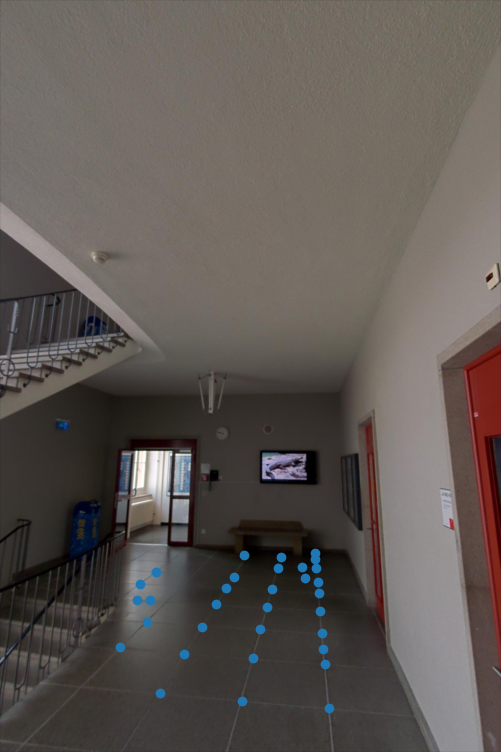}
    \end{tabular}

    \caption{An example of an image from the CAB sequence of the Lamar dataset~\cite{sarlin2022lamar} captured with the NavVis M6 trolley. \textit{left}: the image, \textit{center}: depth detected with a neural network~\cite{depth_anything_v2}, \textit{right}: a plane detected in the image. }
    \label{fig:image_example}
\end{figure}

Here, we present details about the real-world experiments, described in Section~\ref{sec:real} in the main paper. We reuse the notation from that section.

In Algorithm~\ref{alg:detection}, we summarize the process of plane detection used in our experiments. The process is also illustrated in Figure~\ref{fig:image_example}: the first image depicts an image from the Lamar dataset \cite{sarlin2022lamar} used in the experiments, and the subsequent images illustrate the process of plane detection: we show the depth detected using~\cite{depth_anything_v2} used for the plane detection, and the detected plane. The whole experimental pipeline is presented in Algorithm~\ref{alg:experiment}.

\begin{algorithm}[H]
\caption{Plane detection}
\label{alg:detection}
\textbf{Input:} Images $\mathcal{I}_1,\ldots,\mathcal{I}_m$, Camera poses $\mathcal{C} = (C_1,\ldots,C_m)$, correspondences $ \mathcal{U} = x_1,\dots,x_n  \subset (\PP^2)^m $ \\
\textbf{Output:} Detected planes $\Pi_i \in (\PP^3)^*$, and correspondences $\mathcal{U}_i^*$ assigned to them for $i \in \{1,\ldots,num\_planes\}$
\begin{algorithmic}[1]
\STATE $\mathcal{D}_m := $ monocular depth of $\mathcal{I}_m$
\STATE $\mathcal{X} := $ backproject points $x_1,\dots,x_n$ into 3D using $\mathcal{D}_m$
\FOR{$j = 1$ to num\_planes}
    \STATE $\mathcal{X}_j$ := a subset of $\mathcal{X}$ lying on a plane discovered using RANSAC~\cite{ransac}
    \STATE $\mathcal{U}_j$ := correspondences, whose backprojected points lie in $\mathcal{X}_j$
    \STATE $\mathcal{U}_j^2$ := correspondences from $\mathcal{U}_j$ with Sampson error $< 2px$ w.r.t. $\mathcal{C}$
    \STATE $\mathcal{X}_i^2$ := triangulate $\mathcal{U}_j^2$ by fitting into $\mathcal{M}_\mathcal{C}$ (using (P).UC)
    \STATE $\Pi_j$ := fit a plane into $\mathcal{X}_j^2$ using RANSAC~\cite{ransac}
    \STATE $\mathcal{X} := \mathcal{X} \setminus \mathcal{X}_j$
\ENDFOR
\RETURN $\{(\Pi_j, \ \mathcal{U}_j) \ \vert \ j \in \{1,\ldots,num\_planes\}\}$
\end{algorithmic}
\end{algorithm}
\vspace{1em}
\begin{algorithm}[H]
\caption{Plane-constrained triangulation (Experimental pipeline)}
\label{alg:experiment}
\textbf{Input:} Images $\mathcal{I}_{all} = \mathcal{I}_1,\ldots,\mathcal{I}_M$, Camera poses $\mathcal{C}_{all} = (C_1,\ldots,C_M)$, correspondences $\mathcal{U}_{all}$, (Ground truth 3D points $\mathcal{X}_{GT}$) \\
\textbf{Output:} Errors $\mathcal{E}_{(P).UC}$, $\mathcal{E}_{(P).C}$, $\mathcal{E}_{(P).H}$ for each approach
\begin{algorithmic}[1]
\STATE  $\mathcal{E}_{(P).C} := \emptyset, \mathcal{E}_{(P).UC} := \emptyset, \mathcal{E}_{(P).H} := \emptyset$
\FOR{$i=1$ to $M-m+1$}
    \STATE $\mathcal{C} := (C_i,\ldots,C_{i+m-1})$
    \STATE $\mathcal{U}$ := collect point correspondences from $\mathcal{U}_{all}$ observed by cameras $\mathcal{C}$
    \STATE $\{(\Pi_j, \ \mathcal{U}_j) \ \vert \ j \in \{1,\ldots,num\_planes\}\}$ := detect planes using Algorithm~\ref{alg:detection}
    \FOR{$j=1$ to num\_planes}
    \STATE $X_{(P).C}$ := triangulate $\mathcal{U}_j$ by fitting into $\mathcal{M}_{\mathcal{C}}^{\Pi_j}$
    \STATE $X_{(P).UC}$ := triangulate $\mathcal{U}_j$ by fitting into $\mathcal{M}_{\mathcal{C}}$
    \STATE $X_{(P).H}$ := (Reprojection error of $X_{(P).C} < 5px$) ? $X_{(P).C}$ : $X_{(P).UC}$
    \STATE $\epsilon_{(P).C}, \epsilon_{(P).UC, \epsilon_{(P).H}}$ := triangulation errors of $X_{(P).C}, X_{(P).UC}, X_{(P).H}$
    \STATE Add $\epsilon_{(P).C}, \epsilon_{(P).UC}, \epsilon_{(P).H}$ into the corresponding $\mathcal{E}_{(P).C}$, $\mathcal{E}_{(P).UC}$, $\mathcal{E}_{(P).H}$
    \ENDFOR
\ENDFOR
\end{algorithmic}
\end{algorithm}

\printbibliography

\end{document}